\documentclass[11pt]{article}
\usepackage{amsmath,epsfig,graphicx,color, float,amssymb}
\usepackage{amsthm} 
\usepackage{subcaption}

\usepackage{relsize}
\usepackage{comment}
\usepackage[colorlinks, linkcolor=blue,citecolor=blue]{hyperref}
\usepackage{pdfpages}
\usepackage{graphicx}
\usepackage{mwe}
\usepackage{algorithm}
\usepackage{algorithmicx}
\usepackage{algpseudocode}
\usepackage{dsfont}
\usepackage{authblk}

\usepackage[margin=1.3in]{geometry}

\numberwithin{table}{section}
\numberwithin{figure}{section}
\numberwithin{equation}{section}

\definecolor{darkblue}{rgb}{.2, 0.2,.8}
\definecolor{darkgreen}{rgb}{0,0.5,0.3}
\definecolor{darkred}{rgb}{.8, .1,.1}

\newcommand{\dd}{\mathrm{d}}

\newtheorem{lemma}{Lemma}[section]
\newtheorem{theorem}[lemma]{Theorem}

\newtheorem{assumption}{Assumption}

\newtheorem{remark}{Remark}[section]

\usepackage{bm}

\allowdisplaybreaks

\begin{document}
\title{Hybrid Risk Processes: A Versatile Framework for Modern Ruin Problems}

\author[1]{Oscar Peralta\thanks{\href{mailto:oscar.peralta@itam.mx}{oscar.peralta@itam.mx}}}
\author[1]{Habacuq Vallejo\thanks{\href{lvallej6@itam.mx}{lvallej6@itam.mx}}} 

\affil[1]{Autonomous Technological Institute of M\'exico, Department of Actuarial and Insurance Sciences}
\date{}

\maketitle
\begin{abstract}
We introduce the \textit{hybrid risk process}, constructed via a time-change {transformation} applied to the solution of a hybrid stochastic differential equation. The framework covers several modern ruin settings, incorporating features like Markov-modulation and reserve-dependent parameters through an interdependent structure where the surplus level influences the dynamics of the background environment. The approach lets us define and analyze the \textit{Generalized Omega ruin model}, a novel definition of insolvency that synthesizes concepts like Erlangian, cumulative Parisian and Omega ruin into a unified competing-risks framework. Finally, we show that the models are computationally tractable. By adapting recent matrix-analytic techniques \cite{albrecher2022space}, we provide an efficient way to compute a wide range of ruin-related quantities.
\end{abstract}
\section{Introduction}
\label{sec:introduction}

Financial-stability regulation, such as those guided by Solvency II, mandate that insurers hold sufficient capital to absorb unexpected losses and ensure long-term viability. A central principle of these regulations is an explicit constraint on the probability of ruin; for instance, the Solvency II requirement stipulates that this probability {must} not exceed 0.5\% annually. This regulation requires accurate models of the surplus process and its ruin probabilities, which are fundamental inputs for determining capital adequacy and formulating risk management strategies.

The classical Cram\'er-Lundberg model serves as the theoretical foundation of ruin theory. Introduced by Lundberg in 1903 \cite{lundberg1903approximerad} and later placed on a firm probabilistic footing by Cram\'er \cite{cramer1959mathematical}, it describes an insurer's surplus process, $\{U(t)\}_{t \ge 0}$, as
\begin{align}\label{eq:intro1}
U(t) = u + c t - \sum_{i=1}^{N(t)} X_i, \quad t \ge 0,
\end{align}
where $u$ is the initial capital, $c$ is the constant premium rate, $\{N(t)\}_{t \ge 0}$ is a Poisson process counting claims, and $\{X_i\}_{i=1}^\infty$ are the independent and identically distributed claim severities. Within this framework, ruin occurs if the surplus ever falls below zero. Despite its simplifying assumptions, the Cram\'er-Lundberg model's analytical tractability has made it an indispensable tool for understanding long-term solvency.

Since the inception of this seminal model, research in actuarial science has focused on developing more realistic frameworks. Extensions have included generalizing claim arrivals beyond the Poisson process \cite{anderson1957ruin}, introducing dividend strategies \cite{definetti1957dividens, avanzi2009strategies}, incorporating dynamic economic environments \cite{asmussen1989risk}, and allowing premiums to adapt to the current surplus level \cite{Taylor01041980}. Further generalizations have enhanced the underlying stochastic structure, adding diffusion to represent volatility \cite{gerber1970extension} or employing general L\'evy processes to accommodate heavy-tailed claims and market shocks \cite{kluppelberg2004ruin}. In parallel, the definition of ruin has been refined. For instance, some models consider insolvency under discrete monitoring, where ruin is only declared if an observer arriving at Poissonian times finds the surplus to be negative \cite{albrecher2017strikingly}. Other refinements include cumulative Parisian ruin, which allows for a grace period for the total time spent in deficit \cite{guerin2017distribution, bladt2019parisian}, and the Omega model, which distinguishes between a negative surplus and ultimate bankruptcy \cite{albrecher2011gammaomega,gerber2012omega}.

Our contributions are threefold. The first is the introduction of a new modeling framework, the \textit{hybrid risk process}. Constructed via a time-change transformation of a hybrid stochastic differential equation \cite{asmussen2002erlangian,stanford2005phase,bladt2019parisian,peralta2023ruin}, this approach captures the interplay between an insurer's financial state and its operating environment. Its surplus and environment evolve jointly: the surplus evolves in a possibly diffuse, level-dependent manner, while the environment transitions with an intensity that is itself a function of the surplus level. This feedback mechanism, absent in standard Markov-modulated models \cite{asmussen2010ruin}, offers an alternative foundation for more realistic risk models.

Building directly on this framework, our second contribution is the \textit{Generalized Omega ruin} model. By exploiting the structure of the hybrid risk process, we propose a model that unifies and extends several modern ruin definitions. It treats insolvency as the outcome of a competition between different risks, which is modeled using two distinct, level-dependent subintensity matrices that govern the process depending on whether the surplus is positive or negative. This lets us analyze ruin criteria based on multiple factors simultaneously, such as the horizon of the process (as in the Erlangian model), the accumulated duration of its deficit (as in cumulative Parisian ruin) and its severity (as in the Omega model).

A theoretical framework of this generality requires an efficient computational method to be of practical use. Therefore, our third contribution is to demonstrate the tractability of these models. We show that key ruin descriptors can be approximated by adapting the matrix-analytic framework of \cite{albrecher2022space}. The method is based on a spatial discretization of the surplus, which transforms the problem into the analysis of an approximating multi-regime Markov-modulated Brownian motion. This provides a method for computing ruin-related quantities.

The remainder of this paper is organized as follows. In Section \ref{sec:background}, we review the theory of phase-type distributions and hybrid stochastic differential equations. Section \ref{sec:hybrid} details the construction of the hybrid risk process, and Section \ref{sec:model-examples} shows how it integrates many existing models. In Section \ref{sec:ruin-examples}, we formulate various ruin definitions within our framework, leading to the formal introduction of the Generalized Omega model. Section \ref{sec:ruin-hybrid} presents the computational toolbox for calculating ruin descriptors for these processes. Finally, Section \ref{sec:conclusion} provides concluding insights.

\section{Background}\label{sec:background}

This work builds upon two mathematical frameworks. The first is phase-type distributions, which are widely used in actuarial science for their tractability and denseness (see, e.g., \cite{asmussen2010ruin, bladt2017matrix}). The second is the more recent framework of hybrid stochastic differential equations with state-dependent switching, a class of models that is now common because it can capture systems where a diffusion interacts with a finite-state environment. (see, e.g., \cite{yin2010hybrid}). While not traditionally prominent in risk theory under this name, their core feature is particularly relevant for actuarial applications. They describe systems where a continuous process interacts with a finite-state environment where the transition rates of the environment depend explicitly on the current level of the continuous process. This feedback structure is essential to capture realistic dynamics in which economic, regulatory, or operational conditions react to the financial state of an insurer. We review the essential properties of both frameworks that are foundational to our analysis.

\subsection{Phase-type Distributions}

A phase-type (PH) distribution describes the time until absorption of a continuous-time Markov process on a finite state space. More specifically, consider a Markov process on the state space \( \{1, 2, \dots, d, d+1\} \), where states \(1, \dots, d\) are transient and state \(d+1\) is absorbing. The time-homogeneous infinitesimal generator of this process has the form
\[
\begin{pmatrix}
\bm{T} & \bm{t} \\
\bm{0} & 0
\end{pmatrix},
\]
where \( \bm{T} \in \mathbb{R}^{d \times d} \) is a sub-intensity matrix (i.e., it has non-negative off-diagonal entries and non-positive row sums), and the vector of exit rates to the absorbing state is given by \( \bm{t} = -\bm{T} \bm{e} \), with \(\bm{e}\) being a column vector of ones. Given an initial probability distribution \(\bm{\alpha} \in \mathbb{R}^d\) in the transient states, the time to absorption, \( \tau \), follows a phase-type distribution, denoted \(\tau \sim \mbox{PH}(\bm{\alpha}, \bm{T})\). For a comprehensive and modern treatment specialized in this topic, we refer the reader to \cite{bladt2017matrix}.

The class of PH distributions is dense in the set of all distributions on the positive real line, meaning that they can approximate any non-negative random variable arbitrarily well in a weak convergence sense. Their appeal lies in their matrix-analytic structure, which yields closed-form expressions for key quantities. In particular, the probability density function is given by
\[
f(x) = \bm{\alpha} e^{\bm{T}x} \bm{t}, \quad x \geq 0.
\]
In risk theory, the primary application of phase-type distributions is to render classical models analytically and computationally tractable. By assuming that claim sizes or interarrival times follow a phase-type distribution, classical frameworks like the Cramér-Lundberg and Sparre-Andersen models, which are often intractable for general distributions, become solvable using matrix-analytic methods. In these models, the use of PH-distributed components preserves an underlying Markovian structure, which enables the derivation of explicit or semi-explicit solutions for ruin probabilities (see e.g. \cite{asmussen2010ruin}).

A limitation of classical PH distributions is their inherently light-tailed nature, as their survival functions decay exponentially. This makes them unsuitable for modeling heavy-tailed risks. To overcome this, a recent stream of research, initiated in \cite{albrecher2018ph}, has focused on inhomogeneous phase-type (IPH) distributions, leading to further generalizations in \cite{bladt2022heavy}. An IPH distribution corresponds to the absorption time of a time-inhomogeneous Markov process, characterized by an initial distribution \(\bm{\alpha}\) and a time-inhomogeneous subintensity matrix \(\bm{T}(x)\) for $x\ge 0$. The density function is given in terms of a product integral as
\[
f(x) = \bm{\alpha} \left( \prod_{u=0}^x (\bm{I} + \bm{T}(u)\dd u) \right) \bm{t}(x),\quad x \geq 0,
\]
where \(\bm{t}(x) = -\bm{T}(x)\bm{e}\). The product integral term $\prod_{u=0}^x (\bm{I} + \bm{T}(u)\dd u$ represents the matrix of transition probabilities over the interval \([0,x]\), generalizing the matrix-exponential \(\mbox{Exp}(\bm{T}x)\) for a time-varying generator.

The extension from PH to IPH distributions can generate heavy-tailed behavior and offers greater flexibility for fitting complex, real-world data. A particularly useful special case, which is central to the developments in this paper, arises when the sub-intensity matrix has the form \(\bm{T}(x) = c(x)\bm{T}\) for $x\ge 0$, where \(c(x)>0\) is a scalar rate function and \(\bm{T}\) is a constant sub-intensity matrix. In this case, the density simplifies to
\[
f(x) = \bm{\alpha} e^{\bm{T}C(x)} \bm{T} c(x) \bm{e},\quad x \geq 0,
\]
where \(C(x) = \int_0^x c(u) \dd u\) is the cumulative rate function. From a practical perspective, both standard and inhomogeneous phase-type distributions are supported by flexible numerical implementations, such as the \texttt{R} package \texttt{matrixdist} \cite{bladt2021matrixdist}.

\subsection{Hybrid Stochastic Differential Equations}
A hybrid stochastic differential equation (hybrid SDE) describes a continuous process coupled with a finite-state jump process; each influences the other. We denote the solution by the pair \( (J,X)=\{(J(t),X(t))\}_{t\ge0} \), where \(J\) is a finite-state jump process representing the random environment in which the system operates, and \(X\) is a one-dimensional, continuous-path process representing the state of that system. The defining characteristic of the models considered in this paper is the interlaced evolution of their components. The drift and diffusion coefficients of the process \(X\) at time \(t\) depend on the state of the environmental process, \(J(t)\). Simultaneously, the process \(J\) switches between its states with an intensity that is a function of the level of \(X(t)\). More intuitively, for an infinitesimal time interval \(\dd t\), the probability of \(J\) transitioning from state \(i\) to state \(j\) is given by
\[
\mathbb{P}(J(t+\dd t)=j | J(t)=i, X(t)=x) = \delta_{ij} + \Lambda_{ij}(x)\dd t + o(\dd t),
\]
where \(\bm{\Lambda}(x)=\{\Lambda_{ij}(x)\}\) is a level-dependent intensity matrix and \(\delta_{ij}\) is the Kronecker delta. Throughout this paper, the term hybrid SDE will refer to a model with this state-dependent switching feature.

Formally, this system is described by the pair of stochastic differential equations
\begin{align}
\dd X(t) &= \mu(J(t), X(t)) \dd t + \sigma(J(t), X(t)) \dd B(t), \quad X(0) = x_0 \in \mathbb{R}, \label{eq:sde_x_final} \\
\dd J(t) &= \int_{\mathbb{R}_+} h(J(t-), X(t-), z) \mathfrak{p}(\dd t, \dd z), \quad J(0) = i_0 \in \mathcal{E}, \label{eq:sde_j_final}
\end{align}
where \(B\) is a standard Brownian motion and \(\mathfrak{p}(\dd t, \dd z)\) is a Poisson random measure on \(\mathbb{R}_+ \times \mathbb{R}_+\) with intensity measure \(\dd t \dd z\). To ensure the integral representation for \(J\) in \eqref{eq:sde_j_final} is equivalent to the probabilistic description governed by \(\bm{\Lambda}(x)\), the function \(h\) is defined via a uniformization argument. This involves partitioning the mark space of the Poisson measure into regions whose sizes correspond to the entries of the intensity matrix \(\bm{\Lambda}(x)\), such that a random mark falling in a given region triggers the associated jump. For the rigorous details of this construction, we refer the reader to \cite{albrecher2022space}.

For a strong solution of the hybrid SDE \eqref{eq:sde_x_final}--\eqref{eq:sde_j_final} to be well-defined, we impose the regularity conditions set forth in \cite{albrecher2022space}.

\begin{assumption}[Existence of a Strong Solution]
\label{ass:existence_final}
The coefficients of the hybrid SDE satisfy the following conditions.
\begin{enumerate}
    \item[A1.] For each environmental state \(i \in \mathcal{E}\), the drift and diffusion coefficients, \(\mu(i, \cdot)\) and \(\sigma(i, \cdot)\), are Lipschitz-continuous.
    \item[A2.] The family of intensity matrices governing the environmental switching, \(\{\bm{\Lambda}(x)\}_{x \in \mathbb{R}}\), is uniformly bounded.
\end{enumerate}
\end{assumption}

Under Assumptions A1--A2, it is shown in \cite{albrecher2022space} that a strong solution for the pair \((J,X)\) can be constructed explicitly by recursively concatenating the paths of the continuous process between the jump times of the environment. While these particular regularity conditions are chosen because they facilitate a clear construction, more general conditions for the existence of solutions can be found in the literature (see, e.g., \cite{yin2010hybrid}).

A crucial feature of the solution \((J,X)\) is that its second component, representing the surplus, has almost surely continuous paths. Although this framework is powerful for modeling systems with interacting continuous and discrete components, the lack of jumps in the surplus process makes it unsuitable for direct application in most actuarial contexts, where large, instantaneous claim arrivals are a fundamental feature. This has historically limited the use of such models in risk theory.

One of the main contributions of this paper is to bridge this gap. We introduce a modification to the hybrid SDE framework through a time-change transformation that results in a new process that retains the state-dependent structure of the original model while also incorporating the necessary jumps to model claims and other capital flows. This construction, which we term a hybrid risk process, forms the core of our proposed framework and is detailed in the following section.

\section{Hybrid Risk Processes}\label{sec:hybrid}

We start with an underlying hybrid SDE, whose solution we denote by \((J,X)\). We partition the finite state space of $J$ into three disjoint sets,
\[
\mathcal{S}^p \cup \mathcal{S}^+ \cup \mathcal{S}^-,
\]
where \(\mathcal{S}^p\) must be non-empty, while the union \(\mathcal{S}^+ \cup \mathcal{S}^-\) may be empty.
The behavior of the process \(X\) at time $t\ge 0$ is determined by the region in which \(J(t)\) resides, as detailed next.
\begin{itemize}
    \item When the environmental process \(J(t)\) is in a state \(i \in \mathcal{S}^p\), the process \(X\) exhibits general diffusive behavior. We refer to \(\mathcal{S}^p\) as the set of premium-paying states. In these states, the drift \(\mu(i,x)\) and volatility \(\sigma(i,x)\) are without sign or magnitude restrictions; in particular, either or both may be zero.
    \item When \(J(t) \in \mathcal{S}^-\), the volatility vanishes (i.e., \(\sigma(J(t),x)=0\)), and the drift \(\mu(J(t),x)\) is strictly negative. These states will be used to construct the downward jumps corresponding to insurance claims {or other expense events}.
    \item When \(J(t) \in \mathcal{S}^+\), the volatility also vanishes, while the drift \(\mu(J(t),x)\) is strictly positive. These states will be used to construct upward jumps, which can represent capital injections or other income events.
\end{itemize}

We now introduce the hybrid risk process \((L,R)\) through a time-change transformation. First, we define an operational time that only runs when the process is in a premium-paying state,
\begin{equation}\label{eq:theta}
\theta_t = \int_0^t \mathds{1}_{\{J_s \in \mathcal{S}^p\}} \dd s,
\end{equation}
where \(\mathds{1}\) denotes the indicator function. The generalized inverse of this operational time is defined by
\begin{equation}\label{eq:thetaInv}
\theta^\leftarrow_t = \inf\{ s \ge 0 : \theta_s \ge t \}.
\end{equation}
Then, the hybrid risk process \((L,R)\) is defined by evaluating the original process \((J,X)\) at this new time scale, that is,
\begin{equation}\label{eq:hybridRiskProcess}
(L_t, R_t) = ( J_{\theta^\leftarrow_t}, X_{\theta^\leftarrow_t} ).
\end{equation}
The effect of this transformation is to ``delete'' the time intervals where \(J\) sojourns in the drift-only states \(\mathcal{S}^+\) or \(\mathcal{S}^-\). The change in the level of the process \(X\) accumulated during these deleted periods is preserved, however, by being compressed into an instantaneous jump in the new process \(R\). The continuous evolution of \(X\) during the ``active'' periods in \(\mathcal{S}^p\) is directly mapped to the continuous evolution of \(R\). A pathwise illustration of this technique is shown in Figure \ref{fig:path1}.
\begin{figure}[htbp]
 \centering
 \includegraphics[scale=0.33]{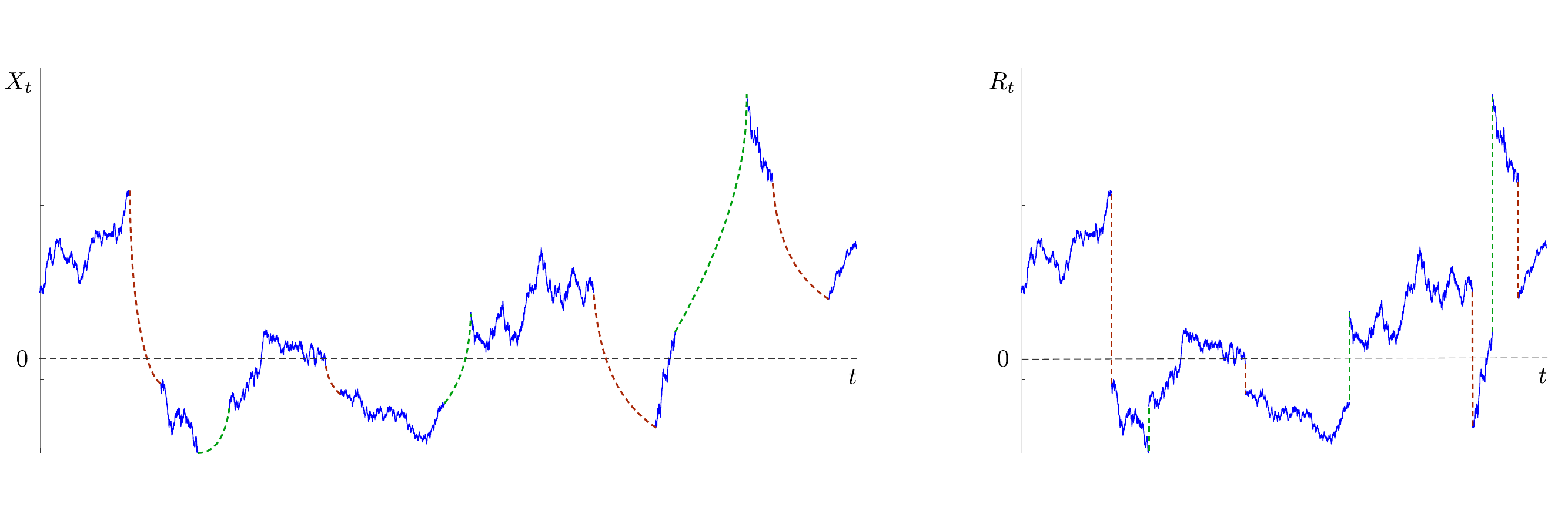}
 \caption{Time-change illustration. Left: A sample path of the process $X$. {Right}: The corresponding path of the transformed process $R$.}
 \label{fig:path1}
\end{figure}

These instantaneous jumps have clear interpretations in risk modeling. A sojourn in \(\mathcal{S}^-\) results in a downward jump in \(R\), representing a claim amount. Conversely, a sojourn in \(\mathcal{S}^+\) results in an upward jump, which can be interpreted as a capital injection. This time-change technique transforms the continuous SDE into a process with jumps, making it suitable for actuarial applications.

\begin{remark}
    The construction method for hybrid risk processes is distinct from standard jump-diffusion models where a separate jump process, such as a compound Poisson process, is typically added to the diffusion equation. The primary advantage of our approach is that the resulting hybrid risk process, \((L,R)\), inherits the analytical structure of the underlying hybrid SDE. Because the framework proposed here leverages space-grid methods for hybrid SDEs (see Section \ref{sec:ruin-hybrid}), we offer a computational framework that, to our knowledge, is not available for generic jump-diffusions.
\end{remark}

To characterize the jumps of the process \(R\), we partition the intensity matrix \(\bm{\Lambda}(x)\) of the process \(J\) according to the state space decomposition \(\mathcal{S}^p \cup \mathcal{S}^+ \cup \mathcal{S}^-\). We impose the structural assumption that a claim event and a capital injection cannot be initiated at the same time. This means that once a sojourn in \(\mathcal{S}^-\) or \(\mathcal{S}^+\) begins, the process must return to a premium-paying state in \(\mathcal{S}^p\) before a jump of the opposite type can occur. This constraint is enforced by setting the corresponding blocks of the intensity matrix to zero, which leads to the structure
\[
\bm{\Lambda}(x) = \begin{pmatrix}
\bm{\Lambda}_{\mathcal{S}^p\mathcal{S}^p}(x) & \bm{\Lambda}_{\mathcal{S}^p\mathcal{S}^+}(x) & \bm{\Lambda}_{\mathcal{S}^p\mathcal{S}^-}(x) \\
\bm{\Lambda}_{\mathcal{S}^+\mathcal{S}^p}(x) & \bm{\Lambda}_{\mathcal{S}^+\mathcal{S}^+}(x) & \bm{0} \\
\bm{\Lambda}_{\mathcal{S}^-\mathcal{S}^p}(x) & \bm{0} & \bm{\Lambda}_{\mathcal{S}^-\mathcal{S}^-}(x)
\end{pmatrix},\quad x\in\mathbb{R}.
\]
The following theorem provides the intensity for the upward and downward jumps of the process \(R\). A crucial consequence of this result is that it establishes a clear link between our time-change construction and the class of IPH distributions, as the jump sizes are shown to follow such a distribution.
\begin{theorem}
\label{thm:jump_density_restructured}
For \(y > 0\) and \(i,j \in \mathcal{S}^p\), the intensity of a downward jump of size approximately \(y\) from level \(x\) is given by
\begin{align*}
&\mathbb{P}(\text{downward jump in }(t, t+\dd t)\text{ of size in }(y, y+\dd y), L_{t+\dd t}=j \mid R_t=x, L_t=i) = \\
&\quad \bm{e}_i^\intercal \bm{\Lambda}_{\mathcal{S}^p\mathcal{S}^-}(x) \left( \prod_{u=0}^y (\bm{I} + \bm{\Delta}_\mu(x-u)\bm{\Lambda}_{\mathcal{S}^-\mathcal{S}^-}(x-u)\dd u) \right) \bm{\Delta}_\mu(x-y)\bm{\Lambda}_{\mathcal{S}^-\mathcal{S}^p}(x-y) \bm{e}_j \dd y \dd t,
\end{align*}
where \(\bm{\Delta}_\mu(z) = \text{diag}(|1/\mu(k,z)|\,:\,k\in \mathcal{S}^-)\) is the diagonal matrix of inverse drift rates for each state \(k \in \mathcal{S}^-\) at level \(z\). An analogous result holds for upward jumps.
\end{theorem}
\begin{proof}
We provide the proof for the downward jump case; the argument for an upward jump is analogous. A downward jump in \((L,R)\) corresponds to an excursion of \((J,X)\) through the state space \(\mathcal{S}^-\). The process unfolds in three steps:
\begin{enumerate}
    \item The process \(J\) transitions from a state \(i \in \mathcal{S}^p\) to an initial state within \(\mathcal{S}^-\). The matrix of intensities for this entry is \(\bm{\Lambda}_{\mathcal{S}^p\mathcal{S}^-}(x)\).
    \item The process \(J\) sojourns in \(\mathcal{S}^-\). During this sojourn, while the environment is in state \(k \in \mathcal{S}^-\), the surplus \(X\) drifts downward at a state- and level-dependent speed given by \(\mu(k,X_s)\). The total accumulated drift over the sojourn duration determines the jump size \(y\).
    \item The process \(J\) exits \(\mathcal{S}^-\) by jumping to a state \(j \in \mathcal{S}^p\).
\end{enumerate}
To find the density of the jump size \(y\), we consider the evolution of the environmental process \(J\) with respect to the accumulated jump size itself, rather than sojourn time. When the accumulated jump size is \(y\) and the environment is in state \(k\), the surplus level is \(x-y\). At this point, the speed at which the jump size increases is given by \(|\mu(k,x-y)|\). To accumulate a further infinitesimal jump of size \(\dd y\), it takes an infinitesimal amount of sojourn time \(\dd s = \dd y / |\mu(k,x-y)|\).

The intensity matrix for transitions within \(\mathcal{S}^-\) with respect to sojourn time \(s\) is \(\bm{\Lambda}_{\mathcal{S}^-\mathcal{S}^-}(x-y)\). To find the intensity matrix with respect to the jump size \(y\), we scale the time-based intensities by the amount of time required to accumulate a unit of jump size. This yields an intensity matrix with respect to \(y\), given by the product
\[
\bm{T}^-(y;x) = \bm{\Delta}_\mu(x-y) \bm{\Lambda}_{\mathcal{S}^-\mathcal{S}^-}(x-y).
\]
The matrix of transition probabilities for the environment over a jump of size \(y\) is then given by the product integral \(\prod_{u=0}^y (\bm{I} + \bm{T}^-(u;x)\dd u)\).

Similarly, the matrix of exit intensities from \(\mathcal{S}^-\) to \(\mathcal{S}^p\) per unit of jump size is \(\bm{\Delta}_\mu(x-y)\bm{\Lambda}_{\mathcal{S}^-\mathcal{S}^p}(x-y)\). Combining the intensity of entry, the propagation during the sojourn (the product integral), and the intensity of exit gives the overall intensity stated in the theorem.
\end{proof}

While this construction provides a rich model, a significant feature is that the jump-size distributions are, by construction, dependent on the surplus level at which they are initiated. This allows for the realistic modeling of heavy-tailed phenomena where the
severity of a shock is linked to the system’s current state. {In a reinsurance context, for instance, the financial consequences of a catastrophic claim are amplified if the insurer's surplus is low. Similarly, regulatory actions, often triggered by low surplus levels, can compound the impact of a large claim, meaning its effective cost distribution is conditional on the insurer's capital at that moment.}

This level-dependency is a flexible feature, not a rigid constraint. If a modeler requires traditional, level-independent claim sizes, this can be achieved by setting all parameters governing the jump to be constant, i.e., making the entry intensities \(\bm{\Lambda}_{\mathcal{S}^p\mathcal{S}^-}(x)\), the drift functions \(\mu(k,x)\) for all \(k \in \mathcal{S}^-\), and the sojourn intensities \(\bm{\Lambda}_{\mathcal{S}^-\mathcal{S}^-}(x)\) and \(\bm{\Lambda}_{\mathcal{S}^-\mathcal{S}^p}(x)\) all independent of the level \(x\). In this scenario, the general matrix distribution for the jump size simplifies to a standard homogeneous phase-type distribution. The ability to accommodate both complex, level-dependent jumps and classic PH jumps (the latter being dense in the space of all positive distributions) highlights the framework's robustness.

Despite the model's richness, the resulting process is (in general) too complex for direct analytical calculations. Therefore, to compute ruin probabilities and other relevant quantities, we will rely on the space-grid approximation scheme, which will be discussed in more detail in Section \ref{sec:ruin-hybrid}.

\section{Examples of Risk Models}\label{sec:model-examples}

In this section, we illustrate how a wide range of modern risk models can be viewed as specific instances of our construction. For each example, we will specify the necessary choices for the components of the underlying hybrid SDE that are required to recover the model in question: the state space partition $\mathcal{S}^p \cup \mathcal{S}^+ \cup \mathcal{S}^-$, the drift and diffusion functions $\mu$ and $\sigma$, and the level-dependent intensity matrix $\bm{\Lambda}(x)$. The models to be considered include Markov-modulated risk processes with diffusion, various dividend payment strategies, models with reserve-dependent premiums, and frameworks that generate heavy-tailed claims.

\subsection{Markov-Modulated Processes with Reserve-Dependent Premiums}

Two of the most significant extensions to the classical Cramér-Lundberg model have been the inclusion of a dynamic economic environment and the introduction of premium rates that depend on the current reserve level. Historically, the analysis of such models developed along two methodological streams. The first, based on direct probabilistic arguments, led to integro-differential equations in \cite{Taylor01041980}, culminating in a framework capable of handling both Markov modulation and state-dependent premiums via systems of ordinary differential equations (ODEs) in \cite{asmussen1996phase}. The second stream of research developed from the analogy between risk processes and fluid queues, an approach that has proven fruitful in various contexts, from analyzing risk in Markovian environments \cite{asmussen1989risk} to approximating finite-horizon ruin probabilities \cite{asmussen2002erlangian, badescu2005risk} and solving other complex ruin problems \cite{bladt2019parisian}. The method is attractive because of its conceptual simplicity and the well-developed analytical toolbox for fluid models, with key contributions on their stationary distributions and underlying theory \cite{rogers1994fluid,asmussen1995stationary}.

However, the application of the fluid queue approach in risk theory has been historically limited to models where parameters are constant within each environmental state. While more general level-dependent fluid models were known within the queueing community, their application was often restricted to processes with piecewise-constant dynamics, sometimes called multi-regime fluid queues \cite{horvath2017matrix,akar2021transient}, and the full generality needed to handle continuously reserve-dependent parameters was not widely adopted by risk analysts. Our hybrid risk process framework bridges this methodological gap. It leverages a fluid-like, time-change construction that can now rigorously incorporate the continuously level-dependent features that were previously primarily accessible through the ODE-based approach. The following subsections will demonstrate this unifying power by showing how seminal models from both of these research streams can be recovered as special cases of our construction. For these examples, we assume there are no capital injections, so the state space \(\mathcal{S}^+\) is empty.

\subsubsection*{The ODE Approach: State-Dependent Premiums}

Here we focus on \cite{asmussen1996phase}, as it is a framework that enjoys {tractability} via the numerical solution of a system of ODEs. The core of their model is a Markov-modulated Poisson process for the claim arrivals, which is governed by a background Markov jump process \(\{Z_t\}\) on a finite state space. In each state \(j\) of the environment, claims arrive with Poisson intensity \(\beta_j\), the premium is collected at a rate given by a level-dependent function \(p_j(r)\), and individual claim amounts follow a state-dependent phase-type distribution, \(B_j\). This model can be framed as a hybrid risk process by specifying our parameters as follows.
\begin{itemize}
    \item The set of premium-paying states, \(\mathcal{S}^p\), is the state space of the environmental Markov process \(\{Z_t\}\).
    \item For each environmental state \(j \in \mathcal{S}^p\), the claim size distribution \(B_j\) has a phase-type representation \(\mbox{PH}(\bm{\alpha}_j, \bm{T}_j)\). The set of claim states, \(\mathcal{S}^-\), is the disjoint union of the transient state spaces of all these claim distributions.
    \item The drift function incorporates the state- and level-dependency of the premium. For \(j \in \mathcal{S}^p\), we set \(\mu(j,x) = p_j(x)\). For states corresponding to the accumulation of a claim, \(k \in \mathcal{S}^-\), the drift is \(\mu(k,x) = -1\).
    \item The diffusion function is zero for all states, \(\sigma(i,x) = 0\), as the model is non-diffusive.
    \item The intensity matrix \(\bm{\Lambda}(x)\) is independent of the level \(x\). With respect to its block structure, \(\bm{\Lambda}_{\mathcal{S}^p\mathcal{S}^p}\) is the constant generator of the environmental process \(\{Z_t\}\). The block \(\bm{\Lambda}_{\mathcal{S}^p\mathcal{S}^-}\) governs claim arrivals; for an environmental state \(j \in \mathcal{S}^p\), the intensity of transitioning to a claim phase \(k\) associated with distribution \(B_j\) is given by \(\beta_j(\bm{\alpha}_j)_k\). The block \(\bm{\Lambda}_{\mathcal{S}^-\mathcal{S}^-}\) is a block-diagonal matrix containing the sub-intensity matrices \(\bm{T}_j\) for each claim type. Finally, the block \(\bm{\Lambda}_{\mathcal{S}^-\mathcal{S}^p}\) governs the completion of a claim; a claim of type \(j\) terminates and returns the environment to state \(j \in \mathcal{S}^p\) with intensities given by the exit vector \(\bm{t}_j = -\bm{T}_j\bm{e}\).
\end{itemize}
This shows that our framework recovers this advanced, unified model. However, our framework further generalizes it by allowing the intensity matrix \(\bm{\Lambda}(x)\) itself to be level-dependent, a feature not considered in \cite{asmussen1996phase}.

\begin{remark}
The model in \cite{asmussen1996phase} was later generalized in \cite[Section 11.1.3]{bladt2017matrix} by replacing the Markov-modulated Poisson arrival process with a Markovian Arrival Process. This more general model can also be incorporated into our hybrid risk process framework, though the setup of the state space and intensity matrix becomes considerably more involved. For the sake of brevity and clarity, we do not detail this specific embedding here.
\end{remark}

\subsubsection*{The Fluid Queue Approach}

The fluid queue approach, known for its simplicity, has been a cornerstone of modern risk theory. A key example that showcases the power of this method is \cite{bladt2019parisian}, which analyzes risk processes that exhibit dependent components. Focusing on their underlying risk model without a diffusion component, their methodology is based on embedding the risk process into an equivalent fluid queue process. The core of this embedding is the construction of a continuous-path process where instantaneous claim jumps are ``straightened out'' into periods of linear downward drift with a constant slope, typically $-1$. The time spent in these downward-drifting periods corresponds to the size of the claims, while the time spent in the upward-drifting premium-collecting periods corresponds to the inter-arrival times. In their model, both inter-arrival times and claim sizes are phase-type distributed, allowing a general dependency structure to be introduced between them.

This construction fits very naturally as a special case of our hybrid risk process. In fact, the embedding described in \cite{bladt2019parisian} is conceptually identical to our time-change transformation, only that theirs is set up in a simpler and less powerful framework. In our framework, the mapping is as follows.
\begin{itemize}
    \item The set of premium-paying states, \(\mathcal{S}^p\), corresponds to the set of phases for the phase-type inter-arrival time distribution.
    \item The set of claim states, \(\mathcal{S}^-\), corresponds to the set of phases for the phase-type claim size distribution.
    \item The drift function is piecewise constant. For any state \(i \in \mathcal{S}^p\), the drift is the constant premium rate, \(\mu(i,x) = c\). For any state \(k \in \mathcal{S}^-\), the drift represents the constant rate of decrease in the fluid level, which corresponds to setting \(\mu(k,x) = -1\).
    \item The diffusion function is \(\sigma(i,x) = 0\).
    \item The intensity matrix \(\bm{\Lambda}(x)\) is constant. Its block structure, which defines the transitions within and between the inter-arrival and claim phases, captures the dependency structure of the model and is given explicitly in \cite{bladt2019parisian}.
\end{itemize}

This demonstrates that common fluid queue-based models, which include the classical Cramér-Lundberg and Sparre-Andersen models with phase-type components, are a specific subclass of the hybrid risk process where all defining parameters are piecewise constant and not dependent on the reserve level.

\subsection{Markov-Modulated Risk Processes with Diffusion Components}

Another significant generalization in risk theory is the inclusion of a diffusion component to model ambient volatility. Our hybrid risk process framework can seamlessly integrate this feature into the Markov-modulated and level-dependent models discussed previously.

The fluid queue perspective can be naturally extended to incorporate a diffusion component, as is done in \cite{bladt2019parisian}. The construction is analogous to the non-diffusive fluid flow process described in the previous subsection; the only modification is the inclusion of a Brownian motion component during the premium-paying periods. Within our hybrid risk process framework, this simply corresponds to setting the diffusion function \(\sigma(i,x) = \sigma > 0\) for all premium-paying states \(i \in \mathcal{S}^p\), while the state space partition, the drift function \(\mu\), and the intensity matrix \(\bm{\Lambda}\) remain unchanged from the non-diffusive case.

Other authors, such as Bäuerle and Kötter \cite{bauerle2005markov}, have also investigated such Markov-modulated diffusion risk models. This structure is particularly prevalent in financial mathematics for option pricing, where the underlying asset follows a jump-diffusion process with parameters modulated by a Markov chain representing the state of the economy \cite{elliott2007pricing}. These financial models can be readily re-contextualized as risk processes within our setup.

More importantly, our framework allows for the inclusion of diffusion in a fully state- and level-dependent manner. This constitutes a direct generalization of the ODE-based model of Asmussen and Bladt \cite{asmussen1996phase}. Now we can consider a model where, for each environmental state \(j \in \mathcal{S}^p\), the premium is a level-dependent function \(p_j(x)\) and the volatility is also a level-dependent function \(\sigma_j(x)\). In this case, the drift and diffusion parameters of our hybrid risk process are simply set to \(\mu(j,x) = p_j(x)\) and \(\sigma(j,x) = \sigma_j(x)\) for \(j \in \mathcal{S}^p\). The ability to handle simultaneous and interacting level-dependencies in the drift, the volatility, and the environmental switching intensities \(\bm{\Lambda}(x)\) within a single, coherent construction is a key strength of the proposed hybrid risk process.
\subsection{Dividend and Taxation Models}

Including dividend strategies makes risk models more realistic because they capture how surplus is withdrawn. These frameworks can also be interpreted as modeling proportional taxation. The flexibility of the hybrid risk process allows for a unified treatment of many classical and modern dividend strategies, where the choice of the drift \(\mu\) and the intensity matrix \(\bm{\Lambda}(x)\) can be used to capture complex payout rules. Key performance metrics in these models include not only the probability of ruin but also the expected discounted value of all dividends paid until ruin and expected occupation times in certain regions of the state space \cite{gerber2012omega}.

The concept of distributing dividends was famously introduced by de Finetti \cite{definetti1957dividens}. Since then, a major focus of the literature has been on barrier strategies \cite{avanzi2009strategies}. In a classical reflection or barrier strategy, any surplus above a predetermined barrier \(b\) is immediately paid out as dividends. This can be modeled in our framework by defining the dynamics of the underlying hybrid SDE to have a reflecting boundary condition at \(b\). A related concept is the refraction strategy, where dividends are paid out at a specific rate \(\delta\) whenever the surplus is above a barrier \(b\) \cite{gerber2006optimal}. This ``refracts'' the path of the surplus process and is a natural fit for our framework, as it corresponds to a level-dependent drift. Specifically, for a state \(i \in \mathcal{S}^p\), the drift becomes \(\mu(i,x) = c_i(x) - \delta\) for all \(x > b\), where \(c_i(x)>\delta\) is the premium rate paid while in environmental state $i$ at the reserve level $x$.

This is an active area; recent work includes extending these classical problems to a more general setting where the underlying risk process is a spectrally negative Lévy process \cite{avram2004exit}. These extensions cover a wide range of sophisticated control problems, including strategies involving both classical and Parisian reflection barriers \cite{avram2018spectrally}. A particularly interesting stream of modern research considers dividend strategies where decisions are not made continuously but at discrete time points. This is often modeled as periodic decisions or as decisions made at the arrival times of an independent Poisson process, which has been studied extensively \cite{albrecher2011optimal, albrecher2011randomized, albrecher2013randomized, zhao2017optimal, noba2018optimal, cheung2019periodic}. The relevance of such models is further highlighted by recent work on Lévy processes subject to level-dependent Poissonian switching \cite{beelders2025evy}.

We now explicitly detail the construction for a specific strategy where dividend payments are controlled by observations occurring at Poisson times with rate \(\lambda\). Let \(b\) be a fixed barrier. If an observer finds the surplus \(R_t > b\), the insurer begins paying dividends at a constant rate \(\delta\). These payments continue until a subsequent observer finds the surplus \(R_t \le b\), at which point they cease. To model this, we augment the environmental state space. The set of premium-paying states becomes \(\mathcal{S}^p_{aug} = \mathcal{S}^p \times \{0,1\}\), where the second coordinate indicates the dividend mode: \(k=0\) for normal operation (not paying dividends), and \(k=1\) for active dividend payment. The full environmental state space for the hybrid SDE is then \((\mathcal{S}^p \times \{0,1\}) \cup \mathcal{S}^+ \cup \mathcal{S}^-\). The strategy is embedded by defining the following parameters.
\begin{itemize}
    \item The drift function \(\mu((i,k), x)\) for \(i \in \mathcal{S}^p\) depends on the dividend mode. For states in the non-paying block, \((i,0)\), the drift is the original premium rate, \(\mu((i,0),x) = c_i(x)\). For states in the paying block, \((i,1)\), the drift is reduced by the dividend payout rate, \(\mu((i,1),x) = c_i(x) - \delta\). The dynamics within \(\mathcal{S}^+\) and \(\mathcal{S}^-\) are unaffected by this logic.
    \item The intensity matrix \(\bm{\Lambda}(x)\) becomes level-dependent to control the switching between dividend modes. For each state \(i \in \mathcal{S}^p\), the intensity of transitioning from a non-paying state \((i,0)\) to a paying state \((i,1)\) is \(\lambda\) if \(x > b\), and 0 otherwise. Conversely, the intensity of transitioning from \((i,1)\) to \((i,0)\) is \(\lambda\) if \(x \le b\), and 0 otherwise. The underlying environmental transitions between different states \(i,j \in \mathcal{S}^p\) are preserved within each dividend mode. 
\end{itemize}
This structure precisely models the described strategy, where a Poisson observation triggers a check against the barrier \(b\), which in turn determines the dividend-paying mode and thus the drift of the surplus process.

\begin{remark}
The example above is one of the simplest forms of a discrete-time dividend strategy. The hybrid SDE framework can capture much more complex behavior. For instance, as explored in \cite{beelders2025evy} for Lévy processes, one could have different observation schemes and dividend rules for up-crossings versus down-crossings of the barrier. This too can be modeled with the hybrid risk process approach.
\end{remark}
\subsection{Heavy-tailed models}

An advantage of the hybrid risk process framework is its natural capacity to generate processes with a rich class of jump distributions, including those with heavy tails, which are crucial for modeling catastrophic risks. As established in Theorem \ref{thm:jump_density_restructured}, the size of a jump in the process $R$ is determined by a sojourn of the underlying process $(J,X)$ in a drift-only state space, such as $\mathcal{S}^-$. The resulting jump size distribution is related to an inhomogeneous phase-type (IPH) distribution whose density is explicitly characterized by the product integral formula.

The model's flexibility is best illustrated in a more structured scenario. If we assume that the drift functions are uniform across all states within the jump region (i.e., $\mu(k,u) = \mu^-(u)$ for all $k \in \mathcal{S}^-$) and that the sojourn intensities $\bm{\Lambda}_{\mathcal{S}^-\mathcal{S}^-}$ are constant, the jump size distribution simplifies to a well-known subclass of IPH distributions. In this case, a scalar rate function $c(y;x)=1/(-\mu^-(x-y))$ governs the evolution of the jump, and the generator for the jump size $y$ takes the form $\bm{T}^-(y;x) = c(y;x)\bm{\Lambda}_{\mathcal{S}^-\mathcal{S}^-}$. This is precisely the structure discussed in \cite[Definition 2.5]{albrecher2018ph}.

The rate function $c(y;x) = 1/|\mu^-(x-y)|$ admits a natural interpretation as a conditional hazard rate. In the context of the jump accumulation, $c(y;x)\dd y$ represents the approximate probability that the jump process terminates (i.e., the underlying process $(J,X)$ returns to a premium-paying state) in the next infinitesimal increment of jump size $(\dd y)$, given that the jump has already reached size $y$. This is analogous to the force of mortality in lifetime data analysis. The functional form of this hazard rate with respect to the accumulated jump size $y$ is what ultimately determines the tail behavior of the resulting jump distribution.

This links between the surplus-dependent behavior of a loss event and the tail of the resulting claim size distribution. We illustrate this with two examples based on this tractable IPH structure.
\begin{itemize}
    \item \textbf{Pareto-like tails:} To generate jumps with a heavy, power-law tail, one can specify a linear drift function during the loss accumulation period, such as $\mu^-(u) = -a - bu$ for $a,b>0$. The drift at the current surplus level $u=x-y$ is then $\mu^-(x-y) = -a -b(x-y)$. The resulting rate function for the IPH is the reciprocal of its magnitude,
    \[ c(y;x) = \frac{1}{a+b(x-y)}. \]
    A rate function that decreases as the reciprocal of the remaining distance to a boundary is characteristic of processes with power-law decay. This choice leads to a matrix-Pareto distribution, as detailed in \cite[Section 3]{albrecher2018ph}.

    \item \textbf{Weibull-type tails:} To generate jumps with Weibull-type tails, one can specify a non-linear drift function of the form $\mu^-(u) = -1/(\beta u^{\beta-1})$. The drift at level $u=x-y$ is $\mu^-(x-y)=-1/(\beta (x-y)^{\beta-1})$, and the corresponding rate function is
    \[ c(y;x) = \beta (x-y)^{\beta-1}. \]
    This rate function is a power function of the remaining distance $(x-y)$, which is the characteristic form of a Weibull hazard rate. This choice directly imparts Weibull-like tail behavior to the jump-size distribution. This construction is detailed in \cite[Section 4.1]{albrecher2018ph}.
\end{itemize}
The ability to generate a wide class of jump distributions connects our framework to a broader research effort into so-called matrix-distributions. While several classes, such as matrix-Pareto and matrix-Weibull distributions, have been proposed to model heavy-tailed phenomena, a recent and robust unification of these ideas can be found in \cite{bladt2022heavy}. Some of the models discussed therein can be recovered with specific choices for the drift functions and sojourn intensities. The key advantage of our construction is that it provides a direct, probabilistic link between the assumed dynamics of the underlying continuous process and the resulting jump size distribution, offering a convenient vehicle for heavy-tailed risk models.

\section{Examples of Types of Ruin}\label{sec:ruin-examples}

Beyond specifying the dynamics of the surplus, a crucial aspect of any risk model is the definition of the ruin event itself. The classical definition, where ruin occurs the first instant the surplus becomes negative, has been extended in many ways to better reflect the operational realities of an insurance company. The flexibility of the hybrid risk process framework accommodates these modern, more nuanced definitions of financial distress, typically by framing them as first-passage time problems for the underlying process $(J,X)$.

This section illustrates the framework's versatility by showing how it can be used to analyze a wide range of insolvency definitions, from classical ruin to more recent and complex criteria. We will demonstrate how each type of ruin can be formulated as a first-passage time problem, often on an augmented state space, making them amenable to a unified computational approach.

We begin by covering the classical definitions distinguished by the time horizon: the infinite-horizon ruin probability and the finite-horizon ruin probability, which is made tractable by considering an Erlang-distributed time horizon \cite{asmussen2002erlangian}. We then discuss ruin types based on different monitoring schemes, such as Poissonian ruin, where insolvency is triggered only if an observer, arriving at times dictated by a Poisson process, finds the surplus to be negative \cite{albrecher2017strikingly}.

The subsequent examples explore more nuanced definitions of financial distress. We consider cumulative Parisian ruin, which posits that insolvency occurs only if the total accumulated time the surplus spends below zero exceeds a grace period, which can itself be random and follow a phase-type distribution \cite{guerin2017distribution, bladt2019parisian}. We also cover the Omega model, which distinguishes between the event of ruin (a negative surplus) and ultimate bankruptcy, with the latter occurring at a rate dependent on the deficit's severity \cite{albrecher2011gammaomega,gerber2012omega}.

Finally, building on these ideas, we introduce a novel definition which we term Generalized Omega ruin. This framework synthesizes concepts from both cumulative Parisian and Omega ruin into a single, unified model of financial distress. In this setup, once the surplus becomes negative, the insurer is subject to termination, which occurs at a single, state- and level-dependent intensity. The key innovation is that this termination rate can be flexibly constructed to depend simultaneously on factors such as the accumulated duration of the process, the time spent in deficit and the severity of that deficit. This approach allows for a more nuanced description of financial distress and naturally contains the Erlangian, cumulative Parisian and Omega models as special cases.

We analyze different ruin definitions for a general hybrid risk process $(L,R)$. As detailed in Section \ref{sec:hybrid}, this process is constructed from an underlying hybrid SDE, $(J,X)$, which is fully characterized by its drift $\mu(i,x)$, diffusion $\sigma(i,x)$, and the environmental generator $\bm{\Lambda}(x)$, for which we assume the conditions in Assumption \ref{ass:existence_final} hold. We will show how various ruin types can be studied by appropriately modifying these components (primarily the generator $\bm{\Lambda}(x)$ via state-space augmentation) to transform each problem into a standard first-passage time problem for a process with killing. This approach allows us to handle complex ruin conditions within the same computational framework.

\subsection{Infinite-horizon Ruin}

The classical and most fundamental definition of insolvency is the infinite-horizon ruin probability \cite{asmussen2010ruin}. Ruin is said to occur if the surplus process ever falls below zero. For our hybrid risk process \((L,R)\), the time of infinite-horizon ruin is defined as
\[
\tau^\infty = \inf\{ t \ge 0 \,:\, R_t < 0 \},
\]
with the convention that the infimum of an empty set is \(\infty\). While the simplest quantity of interest is the probability of ultimate ruin, our framework allows for the analysis of more detailed descriptors. In particular, we can consider the state of the environment at the time of ruin. We therefore define the infinite-horizon ruin probability, conditional on the initial state \((u,i)\) and the state at ruin \(j\), as
\[
\psi_{ij}^\infty(u) = \mathbb{P}(\tau^\infty < \infty, L_{\tau^\infty} = j \mid R_0 = u, L_0 = i).
\]
The subscript \(j\) provides more detail than just the occurrence of ruin; it describes the manner in which the down-crossing of the zero boundary occurs. For instance, if ruin is caused by a jump, \(j\) denotes the state within the jump-generating space $\mathcal{S}^-$ that the process occupied as it crossed the boundary. Throughout this document, we will use the expression $L_{\tau^\infty} = j$ as a convenient, albeit slight, abuse of notation to carry this information about the cause of ruin. The total probability of ruin from state \(i\) is then obtained by summing these quantities over all possible states \(j\) at ruin.

This represents the most straightforward type of ruin to analyze within our framework. The ruin probability \(\psi_{ij}^\infty(u)\) can be computed directly by applying the computational methods discussed in Section \ref{sec:ruin-hybrid}, which are based on the results for first-passage times in \cite{albrecher2022space}. In contrast, the more complex ruin definitions presented in the following subsections are handled by first modifying the underlying hybrid risk process, typically through state-space augmentation, to transform the problem into an equivalent infinite-horizon ruin problem.

\subsection{Erlangian-horizon Ruin}

While the infinite-horizon ruin probability is of theoretical importance, for practical purposes such as solvency regulation, the probability of ruin within a fixed, finite time horizon, \(\psi(u,T)\), is often the quantity of interest. However, direct calculation of \(\psi(u,T)\) is notoriously difficult for most risk models. A powerful and tractable alternative is to approximate the deterministic time horizon \(T\) with an independent, random time horizon \(H\) that follows an Erlang distribution \cite{asmussen2002erlangian}. An Erlang random variable \(H\) with \(m\) stages and rate \(\lambda\), having mean \(\mathbb{E}[H] = m/\lambda\), becomes increasingly concentrated around its mean as the number of stages \(m\) increases. By setting the mean to the desired time horizon, \(T = m/\lambda\), and letting \(m \to \infty\), the probability of ruin before time \(H\) serves as an excellent approximation for \(\psi(u,T)\).

To analyze this problem, we must define two related quantities. Let \(H\) be an Erlang-distributed time horizon independent of the risk process \((L,R)\). The first descriptor is the probability of ruin occurring before the Erlang clock rings. We define this probability, which also keeps track of the environment's state at ruin, as
\[
\psi_{ij}^H(u) = \mathbb{P}(\tau^\infty < H, L_{\tau^\infty} = j \mid R_0 = u, L_0 = i).
\]
The second descriptor is the density of the surplus at the horizon time \(H\), given that ruin has not occurred. This tells us where the process is upon survival, and is defined by the density
\[
\psi_{ij}^H(u, y) \dd y = \mathbb{P}(\tau^\infty \ge H, R_H \in (y, y+\dd y), L_H = j \mid R_0 = u, L_0 = i).
\]
The key to calculating these quantities is to convert the finite-horizon problem into an infinite-horizon problem on an augmented state space. This is done by incorporating the Erlang clock into the state space of the underlying environmental process \(J\). Let the original process \(J\) have state space \(\mathcal{S} = \mathcal{S}^p \cup \mathcal{S}^+ \cup \mathcal{S}^-\) and level-dependent intensity matrix \(\bm{\Lambda}(x)\). Let the Erlang horizon \(H\) have \(m\) stages and rate \(\lambda\), corresponding to a phase-type representation with a subintensity matrix \(\bm{K}\) of size \(m \times m\), where \(\bm{K}\) has \(-\lambda\) on the diagonal and \(\lambda\) on the first super-diagonal.

A crucial subtlety arises from the time-change construction of our hybrid risk process. The Erlang clock \(H\) measures the operational time \(t\) of the process \((L,R)\), which only runs when the underlying environmental process \(J\) is in a premium-paying state \(i \in \mathcal{S}^p\). The clock must be paused during the sojourns in the jump-generating states \(\mathcal{S}^+\) and \(\mathcal{S}^-\). Therefore, the generator for the augmented process must be constructed on a block basis.

The augmented state space is the product space \(\mathcal{S} \times \{1, \dots, m\}\). We now define a new underlying hybrid SDE, \((J^H, X^H)\), on this augmented space. The new environmental process \(J^H\) has states \((i,k) \in \mathcal{S} \times \{1, \dots, m\}\). The continuous component \(X^H\) is identical to the original, i.e., \(X^H_s=X_s\) up to the termination of $X^H$ (more details in Remark \ref{Rem:temination-Erlang}). Since the Erlang clock does not directly influence the surplus dynamics, the drift and diffusion coefficients for the augmented process simply ignore the clock's state, and thus,
\[
\mu^H((i,k), x) = \mu(i,x) \quad \text{and} \quad \sigma^H((i,k), x) = \sigma(i,x),\quad i\in\mathcal{S}, k\in \{1, \dots, m\}, x\in\mathds{R}.
\]
The problem of finite-horizon ruin for the original process \((L,R)\) is then equivalent to the problem of infinite-horizon ruin for a new hybrid risk process \((L^H, R^H)\). This new process is obtained by applying the time-change construction from Section \ref{sec:hybrid} to the augmented SDE \((J^H, X^H)\). The generator for the environmental component \(J^H\) is the matrix \(\bm{\Lambda}^H(x)\), which we construct next.

The new intensity matrix \(\bm{\Lambda}^H(x)\) is constructed as follows, leveraging the properties of the Kronecker sum and product.
\begin{itemize}
    \item \textbf{Concurrent Evolution}: When two independent Markov processes run in parallel, the generator of the joint process is the Kronecker sum (\(\oplus\)) of their individual generators. This is the case when \(J(t) \in \mathcal{S}^p\), as the risk environment and the Erlang clock evolve simultaneously.
    \item \textbf{Paused Evolution}: When one process evolves while the other is frozen, the generator of the joint process involves the Kronecker product (\(\otimes\)) with an identity matrix. This is used to pause the Erlang clock when \(J(t)\) enters \(\mathcal{S}^+\) or \(\mathcal{S}^-\), ensuring the clock's stage is remembered but not advanced.
\end{itemize}
Combining these principles, the generator for the environmental process \(J^H\) is given by the block matrix with the following structure.
\[
\bm{\Lambda}^H(x) =
\begin{pmatrix}
\bm{\Lambda}_{\mathcal{S}^p\mathcal{S}^p}(x) \oplus \bm{K} & \bm{\Lambda}_{\mathcal{S}^p\mathcal{S}^+}(x) \otimes \bm{I}_m & \bm{\Lambda}_{\mathcal{S}^p\mathcal{S}^-}(x) \otimes \bm{I}_m \\
\bm{\Lambda}_{\mathcal{S}^+\mathcal{S}^p}(x) \otimes \bm{I}_m & \bm{\Lambda}_{\mathcal{S}^+\mathcal{S}^+}(x) \otimes \bm{I}_m & \bm{0} \\
\bm{\Lambda}_{\mathcal{S}^-\mathcal{S}^p}(x) \otimes \bm{I}_m & \bm{0} & \bm{\Lambda}_{\mathcal{S}^-\mathcal{S}^-}(x) \otimes \bm{I}_m
\end{pmatrix},
\]
where \(\bm{\Lambda}_{\mathcal{S}^p\mathcal{S}^p}(x) \oplus \bm{K} = \bm{\Lambda}_{\mathcal{S}^p\mathcal{S}^p}(x) \otimes \bm{I}_m + \bm{I}_{|\mathcal{S}^p|} \otimes \bm{K}\).

\begin{remark}\label{Rem:temination-Erlang}
Note that because the Erlang generator \(\bm{K}\) is a subintensity matrix, the resulting augmented matrix \(\bm{\Lambda}^H(x)\) is also a subintensity matrix. The defect in the row sums corresponds to the intensity of the Erlang clock being absorbed, which signifies that the time horizon \(H\) has been reached. This absorption of the clock component effectively terminates the augmented process \((L^H, R^H)\). Formally, this termination can be viewed in two equivalent ways. One is to consider the process as being ``killed'' and sent to an implicit cemetery state, which is the standard interpretation of a subintensity matrix. The other is to explicitly augment the state space further with an absorbing ``horizon'' state. For ease of exposition in this and the following subsections, we adopt the first viewpoint, working with subintensity matrices and treating termination events as absorption into an implicit cemetery state.

It is crucial to analyze the meaning of this termination on a case-by-case basis for each ruin definition. In the present context of Erlangian-horizon ruin, absorption by the clock's exit rates means the time horizon \(H\) has been reached before classical ruin at level zero. Therefore, if this killing occurs while the surplus \(R^H\) has remained in \([0, \infty)\), the probability mass associated with the defect contributes to the non-ruin event described by \(\psi_{ij}^H(u, y)\). As we will see, for other models like Poissonian or cumulative Parisian ruin, the defect in the generator will represent the ruin event itself. The interpretation of the subintensity matrix's defect is thus specific to the problem being modeled.
\end{remark}

The analysis of the Erlang-horizon ruin time begins by setting the initial state of the augmented process. If the original process $(L,R)$ starts with surplus \(u\) and in environment state \(i \in \mathcal{S}\), the augmented process $(L^H,R^H)$ starts in state \((i,1)\) with the same surplus, reflecting that the Erlang clock begins in its first stage. The ruin descriptors are then recovered by analyzing the competition between ruin at level 0 and absorption by the clock. The probability of ruin before the horizon, \(\psi_{ij}^H(u)\), is the probability that the process \((L^H, R^H)\) hits level 0 before the clock component is absorbed. The distribution of the surplus at the horizon, \(\psi_{ij}^H(u, y)\), is the distribution of the surplus at the moment of this absorption (which occurs with intensity \(\lambda\) from any state \((i,m)\)), conditioned on it happening before ruin. With this construction, the Erlangian-horizon problem is transformed into a standard first-passage problem for a process with killing, which can be solved with the computational machinery for our framework.

\subsection{Poissonian Ruin}

The concept of Poissonian ruin moves away from the assumption of continuous monitoring of the surplus process. Instead, insolvency is only declared if the surplus is found to be negative at specific, discrete moments in time. This framework, which has been explored in various contexts such as bridging discrete and continuous-time models and analyzing general Lévy processes \cite{albrecher2016exit, albrecher2017strikingly}, is particularly relevant for modeling regulatory oversight or audits. In this setup, the observation epochs are modeled as the arrival times of an independent Poisson process with some rate \(\lambda>0\).

Let \(\{H_n\}_{n \ge 1}\) be the sequence of arrival times of such a Poisson process. The time of Poissonian ruin, \(\tau^P\), is the first observation time at which the surplus process \(R\) is negative. Formally, this is defined as
\[
\tau^P = \inf\{ H_n \,:\, n \ge 1, R_{H_n} < 0 \}.
\]
In addition to the probability of ruin itself, it is often of interest to analyze the severity of ruin, i.e., the deficit at the time of insolvency. We can therefore define a more detailed descriptor that captures both the occurrence of ruin and the size of the deficit. The Poissonian ruin function is then defined by the density
\[
\psi_i^P(u, y) \dd y = \mathbb{P}(\tau^P < \infty, -R_{\tau^P} \in (y, y+\dd y) \mid R_0 = u, L_0 = i),
\]
which represents the probability of ruin occurring with a deficit of approximately \(y\), given an initial surplus \(u\) and environmental state \(i\).

This type of ruin is handled by analyzing a killed version of the original process. Following the convention established in the previous sections, we describe the dynamics of the process prior to ruin through a level-dependent subintensity matrix. To formalize this, we define a new hybrid risk process, \((L^P, R^P)\), which evolves identically to \((L,R)\) but is subject to this new termination rule. This process is generated from an underlying hybrid SDE, \((J^P, X^P)\), whose drift and diffusion coefficients are unchanged, i.e., \(\mu^P = \mu\) and \(\sigma^P = \sigma\). The key difference lies in the dynamics of the environmental process \(J^P\), which are governed by a new level-dependent subintensity matrix \(\bm{\Lambda}^P(x)\) that incorporates the killing effect of the Poisson observations.

An observation arrives at rate \(\lambda\), but this observation clock is only active when the underlying environment \(J(t) \in \mathcal{S}^p\). Ruin occurs if an observation finds the surplus to be negative. Therefore, the process is subject to a killing intensity that is non-zero only under these specific conditions. This logic defines the subintensity matrix \(\bm{\Lambda}^P(x)\) as follows.
\begin{itemize}
    \item For \(x \ge 0\), no observation can trigger ruin. Thus, the process evolves without killing, and the generator is the original intensity matrix, \(\bm{\Lambda}^P(x) = \bm{\Lambda}(x)\).
    \item For \(x < 0\), an observation while in a state \(i \in \mathcal{S}^p\) triggers ruin. The process is therefore killed with intensity \(\lambda\). This is reflected by modifying the diagonal entries for these states such that \(\bm{\Lambda}^P_{ii}(x) = \bm{\Lambda}_{ii}(x) - \lambda\). For states \(i \notin \mathcal{S}^p\), the clock is paused, so there is no killing, and the corresponding rows of \(\bm{\Lambda}^P(x)\) are identical to those of \(\bm{\Lambda}(x)\).
\end{itemize}
This construction reframes the problem. We are no longer solving a first-passage time problem to the boundary at zero, but are instead interested in the state of the process \((L^P, R^P)\) at the moment of termination. This termination can only occur while the surplus \(R^P_t\) is in the region \((-\infty, 0)\). The quantity we seek, \(\psi_i^P(u, y)\), is precisely the density of the process's location within this negative region at the random time of killing. The probability mass associated with the defect in \(\bm{\Lambda}^P(x)\) for \(x<0\) directly contributes to this ruin function. Calculating this density can be achieved using an adaptation of the methods in \cite{albrecher2022space}, as will be detailed further in Section \ref{sec:ruin-hybrid}.

\subsection{Cumulative Parisian Ruin}
In this section, we analyze a more nuanced definition of insolvency that grants an insurer a degree of resilience to transient deficits. Instead of declaring ruin the instant the surplus becomes negative, we envision the insurer possessing a ``solvency allowance'' that is depleted over time while the surplus is negative. This allowance is modeled as a single, random period of time, which we assume follows a flexible phase-type distribution.

The mechanism works as follows: a timer, corresponding to this allowance, starts running whenever the surplus process is below zero. If the surplus recovers and becomes non-negative, this timer simply pauses. It resumes from where it left off if the surplus drops below zero again at a later stage. Ruin is only declared if the total accumulated time spent in deficit completely exhausts the initial random allowance. This concept is known in the literature as cumulative Parisian ruin \cite{guerin2017distribution, bladt2019parisian}. It is particularly well-suited to our hybrid risk process framework because the ``on/off'' nature of the timer can be modeled smoothly by a level-dependent generator.

To analyze this type of ruin, we first define the quantities of interest. Let \(A_t = \int_0^t \mathds{1}_{\{R_s < 0\}} \dd s\) be the total time the surplus process has spent in deficit by time \(t\), and let \(H\) be the random grace period drawn from a phase-type distribution \(\mbox{PH}(\bm{\kappa}, \bm{K})\). The time of cumulative Parisian ruin, \(\tau^C\), is the first time \(t\) that this accumulated time exceeds the grace period, i.e., \(\tau^C = \inf\{t \ge 0 \,:\, A_t \ge H\}\). Our goal is to compute the density of the deficit at this time, given by
\[
\psi_i^C(u, y) \dd y = \mathbb{P}(\tau^C < \infty, -R_{\tau^C} \in (y, y+\dd y) \mid R_0 = u, L_0 = i),
\]
which represents the probability of ruin occurring with a deficit of approximately \(y\).

To formalize this, we define a new underlying hybrid SDE, \((J^C, X^C)\), on an augmented state space. The augmented environmental process \(J^C\) has states in \(\mathcal{S} \times \{1, \dots, m\}\), tracking both the original environment and the \(m\) phases of the grace-period clock. The continuous component \(X^C\) is identical to the original \(X\), and the drift and diffusion coefficients are unaffected by the clock's state, i.e., \(\mu^C((i,k), x) = \mu(i,x)\) and \(\sigma^C((i,k), x) = \sigma(i,x)\). The cumulative Parisian ruin problem is then studied by analyzing the new hybrid risk process \((L^C, R^C)\), which results from applying the time-change to \((J^C, X^C)\). The dynamics of \(J^C\) are governed by the level-dependent generator \(\bm{\Lambda}^C(x)\) described below.

The generator \(\bm{\Lambda}^C(x)\) changes its structure depending on the sign of the surplus \(x\), which we detail next.
\begin{itemize}
    \item For \(x < 0\), the clock is active and runs on the operational time of the risk process. This means it advances when the underlying environment \(J(t) \in \mathcal{S}^p\) and is frozen during sojourns in the jump-generating states \(\mathcal{S}^+\) or \(\mathcal{S}^-\). Analogous to the Erlangian-horizon construction, this is achieved with a block matrix generator,
    \[
    \bm{\Lambda}^C(x) =
    \begin{pmatrix}
    \bm{\Lambda}_{\mathcal{S}^p\mathcal{S}^p}(x) \oplus \bm{K} & \bm{\Lambda}_{\mathcal{S}^p\mathcal{S}^+}(x) \otimes \bm{I}_m & \bm{\Lambda}_{\mathcal{S}^p\mathcal{S}^-}(x) \otimes \bm{I}_m \\
    \bm{\Lambda}_{\mathcal{S}^+\mathcal{S}^p}(x) \otimes \bm{I}_m & \bm{\Lambda}_{\mathcal{S}^+\mathcal{S}^+}(x) \otimes \bm{I}_m & \bm{0} \\
    \bm{\Lambda}_{\mathcal{S}^-\mathcal{S}^p}(x) \otimes \bm{I}_m & \bm{0} & \bm{\Lambda}_{\mathcal{S}^-\mathcal{S}^-}(x) \otimes \bm{I}_m
    \end{pmatrix},
    \]
    where \(\bm{\Lambda}_{\mathcal{S}^p\mathcal{S}^p}(x) \oplus \bm{K} = \bm{\Lambda}_{\mathcal{S}^p\mathcal{S}^p}(x) \otimes \bm{I}_m + \bm{I}_{|\mathcal{S}^p|} \otimes \bm{K}\).
    \item For \(x \ge 0\), the clock is frozen. The environmental process \(L\) continues to evolve according to its own dynamics, but the clock's phase \(k\) does not change. The generator for this regime is given by the Kronecker product of the original generator with an identity matrix,
    \[
    \bm{\Lambda}^C(x) = \bm{\Lambda}(x) \otimes \bm{I}_m.
    \]
\end{itemize}
This construction results in a process governed by a subintensity matrix \(\bm{\Lambda}^C(x)\) which has a non-zero defect only when \(x<0\). This defect, arising from the exit rates of the PH generator \(\bm{K}\), corresponds to the intensity of cumulative Parisian ruin. The problem is thus transformed into analyzing the location of the process at the time of this killing. It is insightful to consider the special case where the grace period is exponentially distributed with rate \(\lambda\). This corresponds to a phase-type representation with \(m=1\), \(\bm{\kappa}=1\), and \(\bm{K} = [-\lambda]\). In this scenario, the generator \(\bm{\Lambda}^C(x)\) for \(x < 0\) simplifies, and the resulting model becomes equivalent to the Poissonian ruin model discussed in the previous section.

\begin{remark}
It is worth distinguishing this from standard Parisian ruin \cite{dassios2008parisian}, where the grace-period clock resets at the beginning of each new excursion below zero. Modeling the latter concept does not fit as naturally into our framework because it requires forcing a jump in the environmental process (to re-initialize the clock according to \(\bm{\kappa}\)) immediately upon a down-crossing of the zero boundary. Such a boundary-triggered jump is not a feature of the class of hybrid SDEs considered in this paper. While standard Parisian ruin can certainly be analyzed using related techniques, as in \cite{bladt2019parisian}, the required methodology is somewhat different from the continuous, level-dependent generator approach that defines the flavor of this paper.
\end{remark}

\subsection{Omega Ruin}

The Omega model further refines the notion of financial distress by distinguishing between ruin (the event of the surplus becoming negative) and bankruptcy (the ultimate failure of the insurer) \cite{albrecher2011gammaomega,gerber2012omega,li2018fluctuations}. This is motivated by the observation that an insurance company may continue to operate for some time with a negative surplus. In this framework, bankruptcy does not occur at the instant the surplus drops below zero, but is rather triggered by an external shock that arrives at a rate depending on the current state of the insurer's reserve.

More formally, when the surplus process \(R_t\) is negative, the insurer is exposed to the risk of bankruptcy, which can be modeled as the first arrival of an inhomogeneous Poisson process. The intensity of this process, \(\omega\), is a key feature of the model and can depend on both the current environmental state, \(L_t = i\), and the current surplus level, \(R_t = x\), so that the bankruptcy rate is \(\omega(i,x)\). This allows for a very flexible and realistic setup; for instance, the intensity of bankruptcy could increase as the deficit becomes deeper or if the economic environment deteriorates. This type of ``Omega-killed'' process has been analyzed in detail for Markov additive processes \cite{czarna2018omega}.

Formally, the time of Omega ruin, \(\tau^\Omega\), is the first event time of an inhomogeneous Poisson process whose intensity at time \(t\) is given by \(\omega(L_t, R_t)\) if \(R_t<0\) and is zero otherwise. Our objective is to determine the density of the deficit at this random time of bankruptcy, which we define as
\[
\psi_i^\Omega(u, y) \dd y = \mathbb{P}(\tau^\Omega < \infty, -R_{\tau^\Omega} \in (y, y+\dd y) \mid R_0 = u, L_0 = i).
\]
This quantity describes the probability of the insurer going bankrupt with a deficit of approximately \(y\).

To analyze this, we follow our established convention and model the problem by analyzing a killed version of the original process. We define a new hybrid risk process, \((L^\Omega, R^\Omega)\), which evolves identically to \((L,R)\) but is subject to termination by bankruptcy. This process is generated from an underlying hybrid SDE, \((J^\Omega, X^\Omega)\), whose drift and diffusion coefficients are unchanged (\(\mu^\Omega = \mu\), \(\sigma^\Omega = \sigma\)). The dynamics of the environmental process \(J^\Omega\) are governed by a new level-dependent subintensity matrix, \(\bm{\Lambda}^\Omega(x)\), that incorporates the bankruptcy intensity \(\omega(i,x)\).

The construction of \(\bm{\Lambda}^\Omega(x)\) depends on the sign of the surplus, which is detailed next.
\begin{itemize}
    \item For \(x \ge 0\), the insurer is not exposed to Omega ruin. The process evolves without killing, and the generator is the original intensity matrix, \(\bm{\Lambda}^\Omega(x) = \bm{\Lambda}(x)\).
    \item For \(x < 0\), the insurer is exposed to bankruptcy, but this risk is only active during the process's operational time. The process is therefore killed with the state- and level-dependent intensity \(\omega(i,x)\) only if the environment is in a premium-paying state \(i \in \mathcal{S}^p\). For states in \(\mathcal{S}^+\) or \(\mathcal{S}^-\), the bankruptcy clock is paused. This is reflected by modifying the diagonal entries for states in \(\mathcal{S}^p\) such that \(\bm{\Lambda}^\Omega_{ii}(x) = \bm{\Lambda}_{ii}(x) - \omega(i,x)\), while the rows for states not in \(\mathcal{S}^p\) are left unchanged.
\end{itemize}
This construction reframes the Omega ruin problem. We are no longer solving for a first-passage time to zero, but are interested in the distribution of the surplus at the random time of bankruptcy. The probability mass associated with the defect in \(\bm{\Lambda}^\Omega(x)\) for \(x<0\) directly corresponds to this bankruptcy event, and its analysis falls within our computational framework which will be developed in Section \ref{sec:ruin-hybrid}. It is worth noting that this model directly generalizes the Poissonian ruin case; the latter is recovered simply by setting the bankruptcy intensity to a constant \(\omega(i,x) = \lambda\) for all \(i \in \mathcal{S}^p\).

\subsection{Generalized Omega Ruin}
We now introduce a general framework, which can be understood as a two-sided, level-dependent, and state-dependent termination model. The core idea is that the process evolves according to two distinct level-dependent subintensity matrices: \(\bm{\Lambda}_+(x)\) for when the surplus is non-negative (\(x \ge 0\)), and \(\bm{\Lambda}_-(x)\) for when the surplus is negative (\(x < 0\)).

Unlike the previous models where termination was a single killing rate subtracted from the diagonal, here we envision a more general mechanism. These subintensity matrices are constructed from the original generator \(\bm{\Lambda}(x)\) by ``reweighing'' both its off-diagonal and diagonal entries. A portion of each transition intensity from state \(i\) to \(j\) can be redirected to a termination event, and additionally, a direct, state-dependent killing rate can be applied. This implies the process is constantly under competing forces of transition and termination, arising from multiple sources.

Formally, we define a new hybrid risk process, \((L^G, R^G)\), generated from an underlying SDE, \((J^G, X^G)\), with unchanged drift and diffusion coefficients. The dynamics of the environmental process \(J^G\) are governed by the piecewise generator \(\bm{\Lambda}^G(x)\). A key modeling assumption, consistent with our treatment of all time-dependent ruin conditions, is that termination events can only be triggered during the process's operational time. Therefore, the reweighing functions that generate termination are defined to be non-zero only for states \(i \in \mathcal{S}^p\). The generator is constructed as follows.
\begin{itemize}
    \item For \(x \ge 0\), we set \(\bm{\Lambda}^G(x) = \bm{\Lambda}_+(x)\). For a premium-paying state \(i \in \mathcal{S}^p\), its corresponding row in \(\bm{\Lambda}_+(x)\) is constructed using non-negative reweighing functions \(\omega^+_{ij}(x) \in [0,\Lambda_{ij}(x)]\) (\(i \neq j\)) and \(\omega^+_{ii}(x) \ge 0\), leading to
    \begin{align*}
    \Lambda_{+,ij}(x) &= \Lambda_{ij}(x) - \omega^+_{ij}(x) \quad \text{for } j \in\mathcal{S}^p.
    \end{align*}
    For any state \(i \notin \mathcal{S}^p\), its row in \(\bm{\Lambda}_+(x)\) is identical to the corresponding row in \(\bm{\Lambda}(x)\). The resulting killing rate is \(k^+_i(x) = \sum_{j \in\mathcal{S}^p} \omega^+_{ij}(x)\) if \(i \in \mathcal{S}^p\) and zero otherwise. This termination corresponds to a non-ruin event.
    \item For \(x < 0\), we set \(\bm{\Lambda}^G(x) = \bm{\Lambda}_-(x)\). This matrix is constructed in an analogous way, using reweighing functions \(\omega^-_{ij}(x)\) and \(\omega^-_{ii}(x)\) that may be non-zero only for \(i \in \mathcal{S}^p\). The resulting killing rate, \(k^-_i(x)\), corresponds to a ruin event.
\end{itemize}
This framework allows us to analyze the competing risks inherent in the model by defining four key descriptors. Let \(\tau_0 = \inf\{t \ge 0 \,:\,R^G_t < 0\}\) be the time of first down-crossing. The random termination times, \(\tau_+\) and \(\tau_-\), are defined as the first time the augmented environmental process, \(L^G\), leaves its transient state space \(\mathcal{S}_G\). We distinguish between the two based on the sign of the surplus at that moment with
\[
\tau_+ = \inf\{t \ge 0 \,:\, L^G_t \notin \mathcal{S}_G \text{ and } R^G_t \ge 0\} \quad \text{and} \quad \tau_- = \inf\{t \ge 0 \,:\, L^G_t \notin \mathcal{S}_G \text{ and } R^G_t < 0\}.
\]
Here \(\tau_+\) represents non-ruin termination, while \(\tau_-\) represents ruin by termination. These quantities are best understood by grouping them into two distinct competitions.

First, to analyze the competition on the non-negative half-line, we define the following.
\begin{enumerate}
    \item Define the density of the surplus at the first non-ruin termination that precedes any down-crossing,
    \[
    \psi_{ij}^{+<0}(u, y) \dd y = \mathbb{P}(\tau_+ < \tau_0, R^G_{\tau_+} \in (y, y+\dd y), L^G_{\tau_+} = j \mid R^G_0=u, L^G_0=i), \quad y \ge 0.
    \]
    \item The probability that the process down-crosses zero before being terminated on the positive half-line, we define
    \[
    \psi_{ij}^{0<+}(u) = \mathbb{P}(\tau_0 < \tau_+, L^G_{\tau_0} = j \mid R^G_0=u, L^G_0=i).
    \]
\end{enumerate}
Second, to analyze the ultimate competition between termination on the positive versus the negative half-line, we consider the following.
\begin{enumerate}
\setcounter{enumi}{2}
    \item The density of the deficit at the time of a ruin-by-termination event, given that this occurs before any non-ruin termination,
    \[
    \psi_{ij}^{-<+}(u, y) \dd y = \mathbb{P}(\tau_- < \tau_+, -R^G_{\tau_-} \in (y, y+\dd y), L^G_{\tau_-} = j \mid R^G_0=u, L^G_0=i), \quad y > 0.
    \]
    \item The density of the surplus at the time of a non-ruin termination, given that this occurs before any ruin-by-termination event,
    \[
    \psi_{ij}^{+<-}(u, y) \dd y = \mathbb{P}(\tau_+ < \tau_-, R^G_{\tau_+} \in (y, y+\dd y), L^G_{\tau_+} = j \mid R^G_0=u, L^G_0=i), \quad y \ge 0.
    \]
\end{enumerate}
Together, this set of four descriptors provides a comprehensive toolkit for analyzing the various competing risks within the generalized framework.

The framework allows us to unify all previously discussed ruin concepts and, more powerfully, to model the competition between them. To achieve this, we can define both \(\bm{\Lambda}_+(x)\) and \(\bm{\Lambda}_-(x)\) to be non-trivial subintensity matrices. For example, to calculate the probability of ruin versus surviving until a finite time horizon, one could define \(\bm{\Lambda}_+(x)\) based on an Erlang clock augmentation. The ruin-causing matrix \(\bm{\Lambda}_-(x)\) can, in turn, model various scenarios. For cumulative Parisian ruin, it would be constructed by augmenting the state space with a PH grace-period clock. For the classic Omega model, the construction is more direct: \(\bm{\Lambda}_-(x)\) is formed simply by setting the killing intensity for each state \(i \in \mathcal{S}^p\) to be the bankruptcy rate \(\omega(i,x)\), without further state augmentation.

The most powerful application of the framework is modeling the competition between multiple ruin and non-ruin terminations simultaneously. As a premier example, we construct a model that incorporates Omega ruin and cumulative Parisian ruin within a finite time horizon. This requires augmenting the environmental state space \(\mathcal{S}\) with two clocks: a non-ruin clock for the time horizon (e.g., Erlang with \(m_N\) phases and generator \(\bm{K}_N\)) and a ruin clock for the grace period (e.g., a phase-type clock with \(m_R\) phases and generator \(\bm{K}_R\)). The full state space for the environmental process becomes \(\mathcal{S} \times \{1, \dots, m_N\} \times \{1, \dots, m_R\}\).

Complex scenarios might involve augmenting the environmental state space with both a non-ruin and a ruin clock simultaneously, allowing for a direct analysis of, for example, cumulative Parisian ruin versus survival to an Erlang horizon, on top of a base Omega-style ruin model. To construct the generator \(\bm{\Lambda}^G(x)\) for this doubly-augmented process, we first define a base generator \(\bm{\Lambda}^\Omega(x)\) that already incorporates the direct Omega killing: \(\bm{\Lambda}^\Omega_{ii}(x) = \bm{\Lambda}_{ii}(x) - \omega(i,x)\) for \(i \in \mathcal{S}^p, x<0\), and \(\bm{\Lambda}^\Omega(x) = \bm{\Lambda}(x)\) otherwise. We then add the phase-type clocks for the time horizon (\(\bm{K}_N\)) and grace period (\(\bm{K}_R\)). The structure of \(\bm{\Lambda}^G(x)\) depends on the sign of the surplus, which we detail next.
\begin{itemize}
    \item For \(x \ge 0\), the horizon clock \(\bm{K}_N\) is active, but the Parisian clock \(\bm{K}_R\) is frozen. The generator for the premium-paying states is \(\left(\bm{\Lambda}_{\mathcal{S}^p\mathcal{S}^p}(x) \oplus \bm{K}_N\right) \otimes \bm{I}_{m_R}\). The only source of termination is absorption from \(\bm{K}_N\), corresponding to a non-ruin event (reaching the horizon).
    \item For \(x < 0\), the process evolves according to the base subintensity matrix \(\bm{\Lambda}^\Omega(x)\), while both the horizon clock \(\bm{K}_N\) and the Parisian clock \(\bm{K}_R\) are also active. The generator for the premium-paying states is therefore given by the full parallel evolution \(\bm{\Lambda}_{\mathcal{S}^p\mathcal{S}^p}^\Omega(x) \oplus \bm{K}_N \oplus \bm{K}_R\).
\end{itemize}
The defect of the generator for \(x<0\) has three components: the base Omega killing rate \(\omega(i,x)\), the absorption rate from the Parisian clock \(\bm{K}_R\), and the absorption rate from the horizon clock \(\bm{K}_N\). The first two events represent ruin and can be modeled as absorption into a  ruin state \(\partial_R\). In contrast, absorption from the horizon clock \(\bm{K}_N\) (whether for \(x \ge 0\) or \(x < 0\)) represents survival to the time limit and corresponds to absorption into a non-ruin state \(\partial_N\). This detailed construction allows for the direct analysis of these competing risks. The four quantities \(\psi^{+<0}\), \(\psi^{0<+}\), \(\psi^{-<+}\), and \(\psi^{+<-}\) are precisely the descriptors that quantify the probabilities and distributions associated with each possible outcome. The calculation of this complete set of descriptors can be addressed within the computational framework developed in Section \ref{sec:ruin-hybrid}.

\section{Ruin Descriptors for Hybrid Risk Processes}\label{sec:ruin-hybrid}

Having established the hybrid risk process framework, we now outline the computational procedure for its key descriptors. A notable feature of our approach is that the various ruin problems discussed in Section \ref{sec:ruin-examples} can be analyzed within a unified computational scheme. This procedure, based on the framework for multi-regime Markov-modulated Brownian motions (MMBMs) developed in \cite{albrecher2022space}, involves analyzing { its underlying hybrid SDE solution $(J^G,X^G)$} on a generic finite interval $[c,d]$, where $c \le 0 < d$.

The core idea is to compute the Laplace transforms of stopping time distributions, which is equivalent to analyzing a version of the process that is subject to several competing risks. The process can be absorbed at the boundaries $c$ and $d$. Additionally, it faces internal termination events corresponding to non-ruin outcomes (type `$+$', with state-dependent rate $k_i^+(x)$), ruin outcomes (type `$-$', with state-dependent rate $k_i^-(x)$), and an overall process killing by an independent exponential clock, $e_q$, with rate $q \ge 0$.

A key modeling assumption, consistent with the time-change construction, is that all of these internal killing events are only active during the process's operational time—that is, when the environmental process { of the underlying SDE, $J^G$,} is in a premium-paying state $i \in \mathcal{S}^p$. The overall stopping time is therefore the first of these events to occur, namely, $\tau = \tau_c \wedge \tau_d \wedge \tau_+ \wedge \tau_- \wedge e_q$. This framework is implemented by modifying the process generator. For any level $x$, the generator of the full competing risks process, $\bm{\Lambda}^{G,q}(x)$, is obtained by taking the original generator $\bm{\Lambda}^G(x)$ and subtracting all active killing rates from the diagonal entries of the premium-paying states. For each $i \in \mathcal{S}^p$, the corresponding diagonal entry becomes $\Lambda^{G,q}_{ii}(x) = \Lambda^G_{ii}(x) - q$, while all other matrix entries are unchanged.

This general setup allows us to compute a comprehensive set of descriptors for the hybrid risk process $(L^G, R^G)$, which specialize to the cases presented earlier. { To compute these quantities, we follow a multi-step approach: the problem is first translated to the underlying SDE $(J^G, X^G)$; this SDE is then approximated by a tractable MMBM $(\hat{J}^G, \hat{X}^G)$, which is finally analyzed using an auxiliary queueing process. The target quantities for $(L^G, R^G)$ are defined as follows:}
\begin{itemize}
    \item The expected discounted occupation time in state $j$ over a sub-interval $(a,b)$,
    \[
    O_{ij}(u,q;c,d;a,b)=\mathbb{E}\left(\int_{0}^{\tau}\mathds{1}\{L^G(s)=j, a < R^G(s) \le b\}\dd s \mid i,u \right).
    \]
    \item The probability of hitting the lower boundary $0$ before non-ruin termination,
    \[
    \psi_{ij}^{0 < +}(u, q; d) = \mathbb{P}(\tau_0 < \tau_d \wedge \tau_+ \wedge e_q, L^G_{\tau_0} = j \mid i,u).
    \]
    \item The density of the surplus at the time of a non-ruin termination occuring before hitting the lower boundary 0,
    \[
    \psi_{ij}^{+<0}(u, q; d, y) \dd y = \mathbb{P}(\tau_+ < \tau_0 \wedge \tau_d \wedge e_q, R^G_{\tau_+} \in (y, y+\dd y), L^G_{\tau_+} = j \mid i,u).
    \]
    \item The density of the deficit at the time of a ruin-by-termination event,
    \[
    \psi_{ij}^{-<+}(u, q; c,d, y) \dd y = \mathbb{P}(\tau_- < \tau_c \wedge \tau_d \wedge \tau_+ \wedge e_q, -R^G_{\tau_-} \in (y, y+\dd y), L^G_{\tau_-} = j \mid i,u).
    \]
    \item The density of the surplus at the time of a non-ruin termination,
    \[
    \psi_{ij}^{+<-}(u, q; c,d, y) \dd y = \mathbb{P}(\tau_+ < \tau_c \wedge \tau_d \wedge \tau_- \wedge e_q, R^G_{\tau_+} \in (y, y+\dd y), L^G_{\tau_+} = j \mid i,u).
    \]
\end{itemize}
Note that the descriptors in Section \ref{sec:ruin-examples} are recovered in the limit as $q \to 0$, $c\to-\infty$, and $d\to\infty$.

The computational procedure involves two main steps: approximation and analysis. First, the continuous functions of the SDE $(J^G,X^G)$, $\mu(i,\cdot)$, $\sigma(i,\cdot)$, and the generator $\bm{\Lambda}^G(\cdot)$ are approximated by piecewise constant functions $\hat{\mu}(i,\cdot)$, $\hat{\sigma}(i,\cdot)$, and $\hat{\bm{\Lambda}}^G(\cdot)$ over a discrete space-grid. As established in \cite{albrecher2022space}, this approximation converges pathwise to the true process under the following conditions.

\begin{assumption}[Convergence of the Approximation Scheme] Assume that $\hat{\mu}(i,\cdot)$, $\hat{\sigma}(i,\cdot)$, and $\hat{\bm{\Lambda}}^G(\cdot)$ depend on an implicit parameter and have $\mu(i,\cdot)$, $\sigma(i,\cdot)$, and the generator $\bm{\Lambda}^G(\cdot)$ as their pointwise limits as that parameter goes to infinity.
\label{ass:convergence_final}
\begin{enumerate}
    \item[A3.] The approximation errors for the drift and diffusion functions, $\sup_{x\in\mathbb{R}}|\mu(i,x) - \hat{\mu}(i,x)|$ and $\sup_{x\in\mathbb{R}}|\sigma(i,x) - \hat{\sigma}(i,x)|$, decrease at a specified polynomial rate as the grid becomes finer. 
    \item[A4.] The approximation error for the intensity matrix, $\sup_{x\in\mathbb{R}}||\bm{\Lambda}^G(x) - \hat{\bm{\Lambda}}^G(x)||$, decreases at a specified logarithmic rate as the grid becomes finer. 
    \item[A5.] The intensity matrix function $\bm{\Lambda}^G(x)$ is log-Hölder continuous. 
\end{enumerate}
\end{assumption}

\begin{remark}\label{rem:log-holder}
While the log-Hölder continuity condition on $\bm{\Lambda}^G(x)$ in Assumption 4 is general, it excludes several important models discussed in Section \ref{sec:ruin-examples}, most notably Poissonian and cumulative Parisian ruin. The issue is that the underlying clock mechanisms in these models are either activated or frozen depending on the sign of the surplus, which introduces a discontinuity in the generator $\bm{\Lambda}^G(x)$ at the level $x=0$. From a practical standpoint, this discontinuity can be managed by introducing a small buffer zone $(-\epsilon, \epsilon)$ around zero, over which the transition between the two regimes of the generator is smoothed by a continuous function. While this provides a workable patch, the convergence theory for such an approximation as $\epsilon \to 0$ is not established. The core difficulty is that the current theoretical framework for analyzing space-grid approximations of hybrid SDEs does not cover discontinuous generators. Although it is plausible that the approximation scheme still converges in such cases, a rigorous proof is not yet available and remains an open problem for future research.
\end{remark}

In the second step, the descriptors for the { approximating hybrid risk process $(\hat{L}^G, \hat{R}^G)$ are computed. This is achieved} by analyzing an auxiliary recurrent process, $({ K^{\text{aug}}}, Y^{\text{aug}})$, { which is constructed from the components of the MMBM $(\hat{J}^G, \hat{X}^G)$}. To do so, we first define this auxiliary process on the spatial domain $[c,d]$. The environmental state space is augmented to manage the regeneration cycle, $\mathcal{S}^{\text{aug}} = { \mathcal{S}}_G \cup \mathcal{S}_{abs}$, where ${ \mathcal{S}}_G$ is the { state space of $\hat{J}^G$}, and the set of holding states is $\mathcal{S}_{abs} = \{\partial_{j, \ell} \,:\, j \in { \mathcal{S}}_G, \ell \in \{c,d,+,-\}\} \cup \{\partial_u\}$. Here, a state $\partial_{j, \ell}$ signifies absorption of type $\ell$ from an excursion that was in state $j$, while the special state $\partial_u$ serves as a common pre-start location. The dynamics of this auxiliary process are designed to create a regenerative structure, unfolding in stages for a given starting condition $(i_0, u)$ with $u \in (c,d)$, as detailed next.
\begin{enumerate}
    \item \textbf{Excursion Phase:} The process $({ K^{\text{aug}}}, Y^{\text{aug}})$ evolves with state space ${ \mathcal{S}}_G$, governed by the approximated $q$-killed dynamics with drift $\hat{\mu}(i,y)$ and diffusion $\hat{\sigma}(i,y)$. The excursion runs until the first absorption event occurs. At any point during this phase, while in a premium-paying state from $\mathcal{S}^p$, the process can be terminated by the overall $q$-killing mechanism.

    \item \textbf{Absorption:} Upon termination, the process is transferred to a corresponding holding state where its level $Y^{\text{aug}}$ is frozen.
    \begin{itemize}
        \item If the path hits boundary $c$ or $d$ while in state $j$, ${ K^{\text{aug}}}$ enters $\partial_{j,c}$ or $\partial_{j,d}$ respectively, and the level is frozen at that boundary.
        \item If killed internally by rate $k^+$ or $k^-$ while in state $j$, ${ K^{\text{aug}}}$ enters the corresponding state $\partial_{j,+}$ or $\partial_{j,-}$, and the level is frozen at the point of killing.
    \end{itemize}

    \item \textbf{Regeneration Sequence:} Once in any holding state $\partial_{j, \ell}$ (or after being terminated by the $q$-killing mechanism), a multi-stage reset sequence begins. First, the process remains at its frozen absorption level for an $\mbox{Exp}(1)$ amount of time. It then drifts deterministically with a constant speed of $1$ or $-1$ from the absorption level back to the initial level $u$. Upon reaching $u$, the process is transferred to the common waiting state $\partial_u$, where it remains for another $\mbox{Exp}(1)$ amount of time before finally restarting the excursion by jumping from $\partial_u$ back to the initial state $i_0 \in { \mathcal{S}}_G$.
\end{enumerate}
\begin{figure}[htbp]
 \centering
 \includegraphics[scale=0.38]{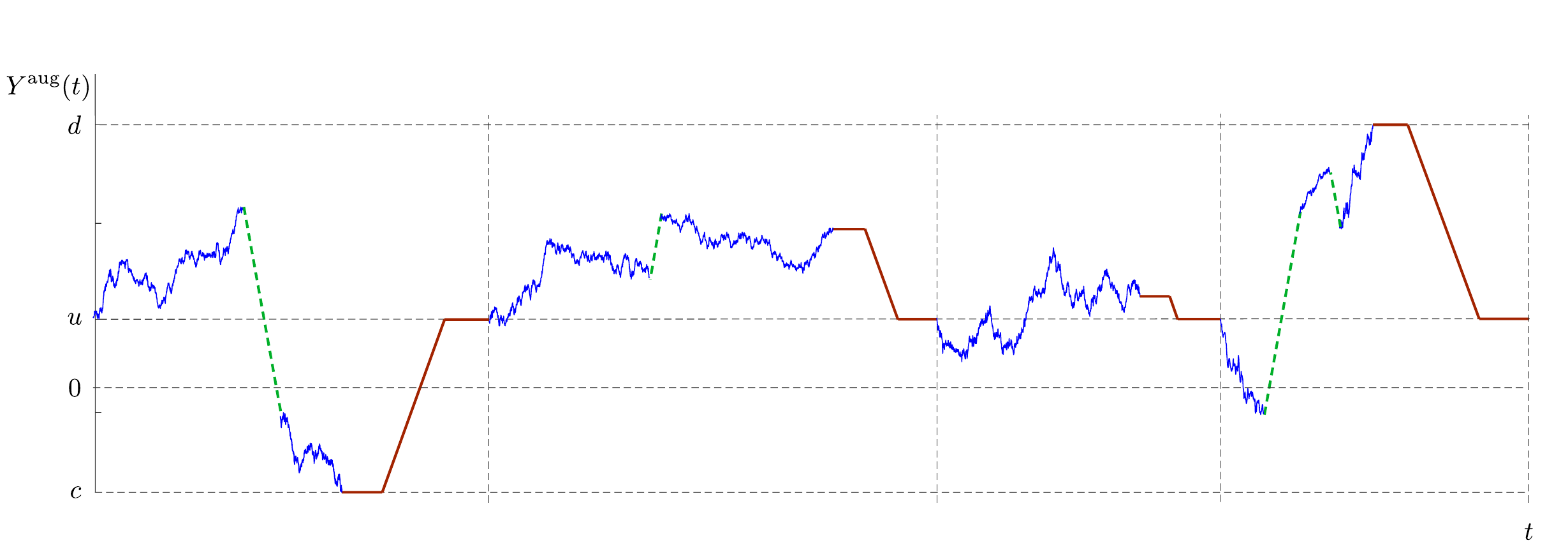}
\caption{A sample path of the auxiliary process $Y^{\text{aug}}$ on the interval $[c,d]$, illustrating the cycle of stochastic excursions and deterministic regenerations. Excursions starting from level $u$ are shown terminating in three ways: by hitting the lower boundary $c$, the upper boundary $d$, and by an internal event. Each absorption triggers an identical regeneration sequence--a holding period, a deterministic drift back to level $u$, and a pre-excursion hold--before a new cycle begins.}
 \label{fig:hybrid-queue}
\end{figure}
See Figure \ref{fig:hybrid-queue} for a visual representation of the aforementioned steps. This construction creates a positive recurrent process whose stationary distribution can be computed efficiently using the matrix-analytic algorithms in \cite{horvath2017matrix, akar2021transient}. The  stationary distribution of $({ K^{\text{aug}}}, Y^{\text{aug}})$ has a mixed nature, consisting of:
\begin{itemize}
    \item A continuous density component, $\boldsymbol{\pi}(y)$, for the process being in an excursion state at level $y \in (c,d)$.
    \item Discrete probability masses (atoms), $\mathbf{p}_c$ and $\mathbf{p}_d$, for the process being held at the boundaries $c$ and $d$, and an atom $p_u$ for being in the pre-start state $\partial_u$.
    \item Continuous density components, $\boldsymbol{\pi}_{\ell}(y)$ for $\ell \in \{+,-\}$, corresponding to the process being in a holding state of type $\ell$ as a result of an internal termination that occurred at level $y$.
\end{itemize}
Following the regenerative process logic from \cite[Theorem 4.1]{albrecher2022space}, we recover the descriptors (denoted with a hat) by normalizing with $p_u$, as summarized in the following theorem.

\begin{theorem}[Descriptors from Stationary Measures]
Let the components of the stationary distribution for the auxiliary process be defined as above. The descriptors for the approximating hybrid risk process $(\hat{L}^G, \hat{R}^G)$, which are the analogous versions of the quantities defined previously for the true risk process $(L^G, R^G)$, are recovered as follows:
\begin{itemize}
    \item The expected occupation time in state $j$ over the interval $(a,b) \subset [c,d]$ is
    \[
    \hat{O}_{ij}(u,q;c,d;a,b) = \frac{1}{p_u}\int_a^b (\boldsymbol{\pi}(y))_j \dd y.
    \]
    \item Taking $c=0$, the probability of absorption at the lower boundary $0$ in state $j$ is
    \[
    \hat{\psi}_{ij}^{0 < +}(u, q; d) = \frac{(\mathbf{p}_{0})_j}{p_u}.
    \]
    \item Taking $c=0$, the probability density of non-ruin termination (type `$+$') at level $y$ is
    \[
    \hat{\psi}_{ij}^{+<0}(u, q; d, y) = \frac{(\boldsymbol{\pi}_{+}(y))_j}{p_u}.
    \]
    \item The probability density of the deficit at ruin-by-termination (type `$-$') at surplus level $y<0$ is
    \[
    \hat{\psi}_{ij}^{-<+}(u, q; c,d, -y) = \frac{(\boldsymbol{\pi}_{-}(y))_j}{p_u}.
    \]
    \item The probability density of non-ruin termination (type `$+$') at level $y$ is
    \[
    \hat{\psi}_{ij}^{+<-}(u, q; c,d, y) = \frac{(\boldsymbol{\pi}_{+}(y))_j}{p_u}.
    \]

\end{itemize}
\end{theorem}
The total probability of an internal termination event can be found by integrating its density or, equivalently, by summing the elements of the corresponding total probability mass vector $\mathbf{p}_\ell = \int_c^d \boldsymbol{\pi}_\ell(y) \dd y$, and normalizing by $p_u$. { These formulas provide the exact descriptors for the approximating risk process $(\hat{L}^G, \hat{R}^G)$, and thus serve as} accurate and computationally efficient estimates for the descriptors of the original hybrid risk process $({L}^G, {R}^G)$.

\section{Conclusions and Future Work}\label{sec:conclusion}

In this paper, we introduced the hybrid risk process, a versatile framework for risk-theory modeling. By applying a time-change transformation to a hybrid SDE with a partitioned state space, we constructed a process that can simultaneously incorporate many of the features seen in modern risk models, including Markov-modulated parameters, level-dependent premiums and volatility, and heavy-tailed jumps. We have demonstrated that our framework is general enough to embed many well-known models from the literature, unifying concepts that were often developed along separate methodological streams, such as those based on ODEs and those based on fluid queues. The recent interest in dividend and taxation strategies, for example, can be readily analyzed within our setup.

Despite its generality, the framework has limits, which stem from its intensity-based nature. The underlying hybrid SDE does not permit the environmental process to switch immediately upon the surplus crossing a specific level. This leaves out models where the jump mechanism is state-dependent in a discontinuous way, for example, the classical Parisian ruin model, where the grace period clock resets instantaneously upon an upcrossing of the zero boundary \cite{bladt2019parisian}. More generally, reliance on the analytical framework of \cite{albrecher2022space} currently requires the level-dependent generator $\bm{\Lambda}(x)$ to be continuous, which excludes important models whose generators are discontinuous, as outlined in Remark \ref{rem:log-holder}.

These limitations suggest two clear avenues for future research. On a practical level, the first-passage quantities computed in this paper can serve as fundamental building blocks to analyze these more specialized processes on a case-by-case basis. By combining these quantities, it may be possible to arrive at manageable and computable solutions. On a more theoretical level, a second, more ambitious project would be to extend the framework of \cite{albrecher2022space} to accommodate hybrid SDEs that allow for state transitions triggered by hitting specific levels, that is, having ``atoms in space''. This would require developing significant new theory but would greatly expand the class of solvable models.

The incorporation of the Generalized Omega model, which can simultaneously exhibit a finite horizon, a cumulative grace period, and a direct bankruptcy rate, serves as an illustration of the flexibility of the framework. However, our framework can also guide the design of new ruin criteria. The ability to specify arbitrary level-dependent subintensity matrices \(\bm{\Lambda}_+(x)\) and \(\bm{\Lambda}_-(x)\) provides a grammar for designing novel and practically relevant ruin criteria tailored to specific regulatory or business contexts. A full exploration of these possibilities is beyond the scope of this paper and is a topic for future work.

Finally, we have outlined how to use the space-grid approximation technique from \cite{albrecher2022space} to efficiently compute important ruin descriptors for a wide variety of models within our framework. By linking the hybrid risk process to a multi-regime MMBM, we can leverage the existing computational tools available for that class of processes, making our framework not only theoretically sound but also practically applicable.
\section*{Acknowledgments}
O. Peralta’s work was supported by the Asociación Mexicana de Cultura, A. C.


\bibliographystyle{alpha}
\bibliography{RiskHybridSDE}

\end{document}